\documentclass[joc]{ipart}

\usepackage{graphicx,amsthm,amsmath, amsfonts}
\usepackage{hyperref}

\pubyear{2017}
\volume{0}
\issue{0}
\firstpage{1}
\lastpage{1}

\startlocaldefs

\newtheorem{theorem}{Theorem}[section]
\newtheorem{proposition}[theorem]{Proposition}
\newtheorem{corollary}[theorem]{Corollary}
\newtheorem{lemma}[theorem]{Lemma}

\newtheorem{example}[theorem]{Example}

\newtheorem{defn}[theorem]{Definition}
\theoremstyle{definition}

\newcommand{\mult}{{\mathrm {mult}}}

\newcommand{\Krew}{{\mathsf {Krew}}}
\newcommand{\krew}{{\mathtt {krew}}}
\newcommand{\Park}{{\mathsf {Park}}}
\newcommand{\rot}{{\mathtt {rot}}}
\newcommand{\cat}{{\mathtt {Cat}}}
\newcommand{\rfn}{{\mathtt {rfn}}}
\newcommand{\abs}{{\mathrm{Abs}}}
\newcommand{\rank}{{\mathrm{rank}}}
\newcommand{\Cat}{{\mathsf{Cat}}}
\newcommand{\Nar}{{\mathsf{Nar}}}
\newcommand{\symm}{{\mathfrak{S}}}
\newcommand{\C}{{\mathbb {C}}}
\newcommand{\Z}{{\mathbb {Z}}}
\newcommand{\N}{{\mathbb{N}}}

\endlocaldefs

\begin{document}

\begin{frontmatter}

\title{Rational Noncrossing Partitions for all Coprime Pairs}
\runtitle{Rational Noncrossing Partitions for all Coprime Pairs}


\author{\fnms{Michelle} \snm{Bodnar}\ead[label=e1]{mbodnar@ucsd.edu}}
\address{Department of Mathematics\\
University of California, San Diego\\
La Jolla, CA, 92093-0112, USA\\
\printead{e1}}

\runauthor{Michelle Bodnar}

\begin{abstract}
For coprime positive integers $a<b$, Armstrong, Rhoades, and Williams (2013) defined a set $NC(a,b)$ of rational noncrossing partitions, a subset of the ordinary noncrossing partitions of $\{1, \ldots, b-1\}$. Bodnar and Rhoades (2015) confirmed their conjecture that $NC(a,b)$ is closed under rotation and proved an instance of the cyclic sieving phenomenon for this rotation action.  We give a definition of $NC(a,b)$ which works for all coprime $a$ and $b$ and prove closure under rotation and cyclic sieving in this more general setting. We also generalize noncrossing parking functions to all coprime $a$ and $b$, and provide a character formula for the action of $\mathfrak{S}_a \times \Z_{b-1}$ on $\mathsf{Park}^{NC}(a,b)$. 
\end{abstract}



\end{frontmatter}

\section{Introduction}
Let $W$ be a Weyl group with root lattice $Q$, degrees $d_1, d_2, \ldots, d_\ell$, and Coxeter number $h = d_\ell$.  Then $W$ acts on the ``finite torus'' $Q/(h+1)Q$.  Cosets in $Q/(h+1)Q$ give a model for parking functions attached to $W$ \cite{ParkingSpaces}.  It has been shown by Haiman \cite{Haiman} that the number of orbits of this action is given by
\[ \cat(W) := \prod_i \frac{h + d_i}{d_i}, \]
which has come to be known as the Coxeter-Catalan number of $W$.  More generally, if $p$ is a positive integer which is coprime to the Coxeter number $h$, Haiman \cite{Haiman} showed that the number of orbits in the action of $W$ on $Q/pQ$ is  
\[ \cat(W,p) = \prod_i \frac{p + d_i - 1}{d_i}.\]
This number has come to be known as the \emph{rational Catalan number} of $W$ at parameter $p$.  

On the level of Weyl groups the Catalan and Fuss-Catalan objects, obtained by taking $p = h + 1$ and $mh + 1$ respectively, have been defined and studied \cite{memoirs}. When $W = \symm_a$ is the symmetric group, we have
\[\cat(\symm_a, a+1) = \frac{1}{a+1}{2a \choose a} = \cat(a)\]
where $\cat(a)$ is the classical Catalan number, famously counting noncrossing partitions, Dyck paths, well-paired parentheses, as well as hundreds of other combinatorial objects.  Furthermore, we have
\[ \cat(\symm_a, ka + 1) = \frac{1}{ka + a + 1} {ka + a + 1 \choose a}  = \cat^{(k)}(a)\]
where $\cat^{(k)}(a)$ is the Fuss-Catalan number, counting generalizations of Catalan objects such as noncrossing partitions whose block sizes are all divisible by $k$.  However, it wasn't until 2013 that Armstrong et. al \cite{RatAss, ArmPark} undertook a systematic study of type A rational Catalan combinatorics.

For coprime positive integers $a$ and $b$, the \emph{rational Catalan number} is 
\[ \cat(\symm_a, b) = \frac{1}{a+b} {a + b \choose a,b} = \cat(a,b).\]
Observe that $\cat(n, n+1) = \cat(n)$, so that rational Catalan numbers are indeed a generalization of the classical Catalan numbers.  The program of rational Catalan combinatorics seeks to generalize Catalan objects such as Dyck paths, the associahedron, noncrossing perfect matchings, and noncrossing partitions (each counted by the classical Catalan numbers) to the rational setting. For instance, $\cat(a,b)$ counts the number of \emph{$a,b$-Dyck paths}, NE-lattice paths from the origin to $(b,a)$ staying above the line $y = \frac{a}{b}x$.  

For coprime parameters $a < b$, Armstrong et. al \cite{RatAss} defined the $a,b$-\emph{noncrossing partitions}, $NC(a,b)$, to be a subset of the collection of noncrossing partitions of $[b-1]$ arising from a laser construction involving rational Dyck paths.  A characterization of these rational noncrossing partitions was given in \cite{RatCat}, where it was shown that $NC(a,b)$ is closed under dihedral symmetries and that the action of rotation on $NC(a,b)$ exhibits a cyclic sieving phenomenon.  Additionally, a model for $a,b$-noncrossing parking functions was given which carries an $\symm_a \times \Z_{b-1}$ action, and a character formula was stated and proved. However, this rational generalization and others rely on the fact that $a < b$. 

It is of intrinsic combinatorial interest to see whether such results hold in the case where $a > b$.  Moreover, this seems reasonable as Haiman's formula holds for any coprime pair $a,b$.  Furthermore, we are motivated by the favorable representation theoretic properties of the rational Cherednik algebra attached to the symmetric group $\symm_a$ at parameter $b/a$.  Such properties persist even when $a > b$.  It is thus desirable to remove the condition $a < b$ and define rational noncrossing partitions for all coprime pairs $(a,b)$. This paper provides the first type $A$ combinatorial model for rational Catalan objects defined for all coprime $a$ and $b$, along with proofs of generalizations of many of the results known for the $a < b$ case.

The rest of the paper is organized as follows:  \textbf{Section 2} begins with background information on rational Dyck paths and noncrossing partitions.  In \textbf{Section 3} we give an interpretation of $a,b$-Dyck paths in terms of pairs of labeled noncrossing partitions which generalizes the laser construction introduced in \cite{RatAss} and prove some basic properties of $NC(a,b)$.  In \textbf{Section 4} we show that $NC(a,b)$ is closed under a suitable rotation action.  We then provide a notion of block rank and show that block rank commutes with rotation.  We also provide a characterization of when a given labeled pair of noncrossing partitions is an element of $NC(a,b)$ and show that this set is closed under a suitable reflection action.  \textbf{Section 5} introduces $d$-modified rank sequences and goes through a series of lemmas which ultimately allow us to count the number of elements in $NC(a,b)$ which are invariant under $d$-fold rotation. In \textbf{Section 6}, we prove various refinements of cyclic sieving results for $NC(a,b)$.  A step in this direction has already been made by Thiel in \cite{Marko} where cyclic sieving was shown in the case $(a,b) = (n+1, n)$ by considering objects called noncrossing $(1,2)$-configurations.  We will revisit this construction and give a bijection between our $n+1,n$-noncrossing partitions and these $(1,2)$ configurations.  In \textbf{Section 7}, we generalize $a,b$-noncrossing parking functions, $\Park^{NC}(a,b)$, to all coprime $a$ and $b$ and prove a character formula for the action of $\symm_a \times \Z_{b-1}$ on $\Park^{NC}(a,b)$.  Finally, \textbf{Section 8} offers a possible direction for future research. 

\section{Background}

\subsection{Rational Dyck Paths}

Let $a$ and $b$ be coprime positive integers.  An $a,b$-\emph{Dyck path} $D$ is a lattice path in $\Z^2$ consisting of unit length north and east steps which starts at $(0,0)$, ends at $(b,a)$, and stays above the line $y = \frac{a}{b}x$.  By coprimality, the path will never touch this line. For example, consider the 7,4-Dyck path $NNNENENNENE$ shown in Figure \ref{vertRun}.  
\begin{figure}[h]
\includegraphics[scale=.4]{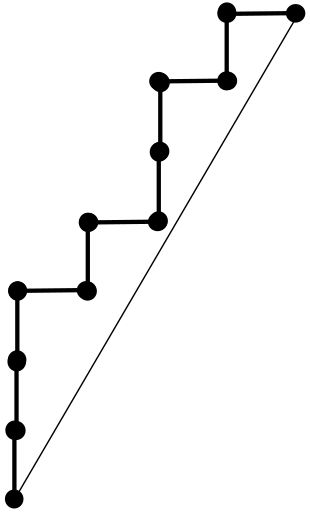}
\caption{7,4-Dyck Path NNNENENNENE with the line $y = \frac{7}{4}x$}
\label{vertRun}
\end{figure}
The $a,b$-Dyck paths are counted by the \emph{rational Catalan number} $\Cat(a,b) = \frac{1}{a+b}{a + b \choose a}$. A \emph{vertical run} of $D$ is a maximal contiguous sequence of north steps.  The 7,4-Dyck path shown in Figure \ref{vertRun} has 4 vertical runs of lengths 3, 1, 2, and 1 respectively.   Note: it is possible for a vertical run to have length 0.  A \emph{valley} of $D$ is a lattice point $p$ on $D$ such that $p$ is immediately preceded by an east step and succeeded by a north step. Figure \ref{vertRun} has three valley points.  When $(a,b) = (n, n+1)$, rational Dyck paths are equivalent to classical Dyck paths, NE-lattice paths from $(0,0)$ to $(n,n)$ which stay weakly above the line $y=x$, and are counted by the classical Catalan numbers.

\subsection{Noncrossing Partitions}

 A set partition $\pi$ of $[n] := \{1, 2, \ldots, n\}$ is \emph{noncrossing} if its blocks do not cross when drawn on a disk whose boundary is labeled clockwise with the number 1, 2, $\ldots$, $n$.  Equivalently, $\pi$ is noncrossing if there do not exist $a < b < c < d$  such that $a$ and $c$ are in the same block $B$, and $b$ and $d$ are in the same block $B' \neq B$.  Let $NC(n)$ denote the set of noncrossing partitions of $[n]$. Such partitions are counted by the classical \emph{Catalan numbers}

\[ \cat(n) = \frac{1}{n+1} {2n \choose n} = \frac{1}{2n+1} {2n+1 \choose n} = |NC(n)|.\]
The \emph{rotation} operator $\rot$ acts on the set $NC(n)$ by the permutation 
\[\left( \begin{array}{ccccc} 1 & 2 & \cdots & n-1 & n \\ 2 & 3 & \cdots & n & 1 \end{array}\right).\]
  A \emph{labeled noncrossing partition} is a noncrossing partition with a nonnegative integer called a label attached to each block.  When we apply $\rot$ to a labeled noncrossing partition the elements of each block shift as in the unlabeled case, and blocks maintain their labels throughout the rotation. 

\section{Construction and Properties of $NC(a,b)$}

\subsection{Rational Pairs of Noncrossing Partitions}

A simple bijection maps classical Dyck paths to noncrossing partitions.  The same map, when $a < b$, sends $a,b$-Dyck paths to $a,b$-noncrossing partitions \cite{RatCat}.  We'll now define a more general version of this map, $\pi$, that makes sense for any $a,b$-Dyck path and use this map to define rational $a,b$-noncrossing partitions for any coprime $a$ and $b$.  Let $D$ be an $a,b$-Dyck path and label the east ends of the nonterminal east steps of $D$ from left to right with the numbers $1, 2, \ldots, b-1$.  Let $p$ be the label of a lattice point at the bottom of a north step of $D$.  The \emph{laser} $\ell(p)$ is the line segment of slope $\frac{a}{b}$ which fires northeast from $p$ and stops the next time it intersects $D$.  By coprimality, $\ell(p)$ terminates on the interior of an east step of $D$.  For instance, consider the 10,7-Dyck path shown on the left in Figure \ref{firstEx}.  We have that $\ell(3)$ hits $D$ on the interior of the east step whose west endpoint is labeled 5.  We define the \emph{laser set} $\ell(D)$ to be the set of pairs $(i,j)$ such that $D$ contains a laser starting at label $i$ and which terminates on an east step with west $x$-coordinate $j$. For the Dyck path in Figure \ref{firstEx} we have 
\[ \ell(D) = \{(1,1), (2,6), (3,5), (4,5), (6,6)\}.\]

Define a pair of labeled noncrossing partitions $\pi(D) = (P,Q)$ as follows:  fire lasers from all labeled points which are also at the bottom of a north step.  We define the partition $P$ by the \emph{visibility relation}
\[ i \underset{P}{\sim} j \mbox{ if and only if the labels $i$ and $j$ are not separated by laser fire.}\]
We make the convention that the label $i$ lies stricly below $\ell(i)$.  Label each block of $P$ by the length of the vertical run immediately preceding the minimal element of the block. We will refer to this as the rank of the block.  Call a vertical run a \emph{$P$-rise} if it has length greater than $\frac{a}{b}$.   We will now describe the creation of the blocks of $Q$, a genuinely new feature of this map.

We call a vertical run a \emph{$Q$-rise} if it has length which is less than $\frac{a}{b}$, including zero.  In the special case where $a<b$, there can only be $Q$-rises of length zero since $a/b < 1$.  In Figure \ref{firstEx}, the vertical runs with $x$-coordinates 0, 2, 3, and 4 are $P$-rises and with $x$-coordinates 1, 5, and 6 are $Q$-rises.  We define the partition $Q$ by the relation 
\[ i \underset{Q}{\sim} j \]
if and only if one of the following holds:
\begin{enumerate}
\item $\ell(i)$ and $\ell(j)$ hit the same east step immediately following a $Q$-rise
\item $(i,j) \in \ell(D)$
\item $(j,i) \in \ell(D)$.
\end{enumerate}
We label the blocks of $Q$ as follows:  If $B$ is a block of $Q$ and $i \in B$, then we label $B$ with the number of north steps beneath the west endpoint of the east step hit by $\ell(i)$.  If $i$ doesn't fire a laser, then we assign $B$ rank 0.  This is well-defined because different elements of a block of $Q$ always touch or fire a laser which hits the same east step.  As with $P$, we will call this block labeling the rank of the block.  There will often be blocks of rank 0, which we will call the \emph{trivial blocks} of $Q$. We will refer to blocks of $Q$ whose ranks are positive as \emph{nontrivial blocks}.  Let $\pi(D)$ denote the labeled pair $(P, Q)$ associated to $D$ under this construction.   

\begin{figure}[h]
\includegraphics[scale=.7]{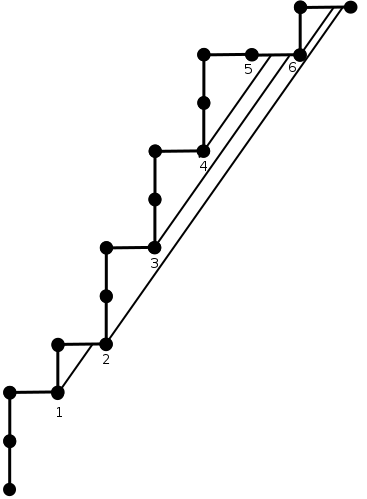}
\includegraphics[scale=.7]{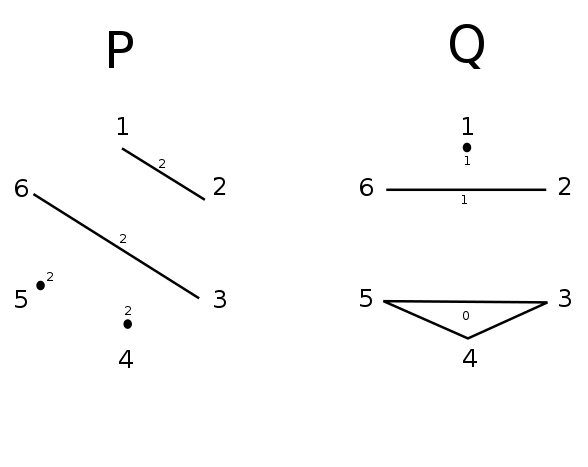}
\caption{A 10,7-Dyck path with corresponding pair of labeled noncrossing partitions}
\label{firstEx}
\end{figure}

Figure \ref{firstEx} shows a 10,7-Dyck path with labels and lasers drawn in.  The pair $(P,Q)$ which results, also shown in Figure \ref{firstEx}, is as follows:  
\[P = \{ \{1,2\}, \{3,6\}, \{4\},\{5\}\}\]
where each block has rank 2.
\[Q = \{ \{1\}, \{2,6\}, \{3,4,5\}\}\]
 with block ranks 1, 1, and 0 respectively.  In particular, the block $\{3,4,5\}$ is a trivial block of $Q$.  The ranks are written in smaller font near the lines indicating the block structure.  We will often omit the trivial blocks of $Q$ and simply write $Q = \{\{1\}, \{2,6\}\}$, each with rank 1.  

Each north step contributes to the rank of either a $P$ block or a $Q$ block, but not both.  In particular, the length of a $P$-rise is the rank of a block of $P$, and the length of a $Q$-rise is the rank of a block of $Q$.  This implies that the sum of the ranks of the $P$ and $Q$ blocks is $a$.  Note that elements in the same block of $Q$ are necessarily in different blocks of $P$, since elements in the same block of $Q$ are always separated by at least one laser.

When $a < b$, $Q$ contains only blocks of rank 0 and $P$ is the rational noncrossing partition associated to $D$ as described by the map in \cite{RatCat}.  The ranks of blocks are uniquely determined in this case by the structure of $P$, which is why labeling blocks by rank has not previously been considered.  When $a > b$, the ranks of a blocks are no longer uniquely determined by the structure of $P$ and $Q$.  For instance, the 5,3-Dyck paths $NNNENNEE$ and $NNENNNEE$ both give rise to $P = \{\{1\}, \{2\}\}$ and only trivial $Q$ blocks. Thus, the rank labels are a necessary feature of the construction of $\pi(D)$.  Since ranks tell us precise vertical run lengths, the map $\pi$ is injective. We are now ready to prove some useful properties of $a,b$-noncrossing partitions.

\begin{proposition}\label{minMax}
Let $(P,Q) = \pi(D)$ for an $a,b$ Dyck path $D$.  There cannot exist $1 \leq i < b-1$ such that $i$ is the maximal element of a block of $Q$ and $i+1$ is the minimal element of a block of $P$. 
\end{proposition}
\begin{proof}
If $i$ is the maximal element of a block of $Q$ then the lattice point labeled $i$ in $D$ is at the bottom of a $Q$-rise, whose the length is less than $a/b$.  On the other, hand if $i+1$ is also the minimal element of a block of $P$ then the lattice point labeled $i$ is at the bottom of a $P$-rise, whose length must be greater than $a/b$, a contradiction.  
\end{proof}

At this point it will be useful to introduce the \emph{Kreweras complement} of a noncrossing partition.  Let $P$ be a noncrossing partition.  The Kreweras complement, denoted $\krew(P)$, is computed as follows:  Begin by drawing the $2n$ labels $1, 1', 2, 2', \ldots, n, n'$ clockwise on the boundary of a disk.  Next, draw the blocks on $P$ on the unprimed vertices.  Then $\krew(P)$ is the unique coarsest partition of the primed vertices which introduces no crossings.  An example is shown in Figure \ref{krew}.  The map $\krew : NC(n) \to NC(n)$ satisfies $\krew^2 = \rot$, so $\krew$ is a bijection.   

\begin{figure}[h]
\includegraphics[scale=.4]{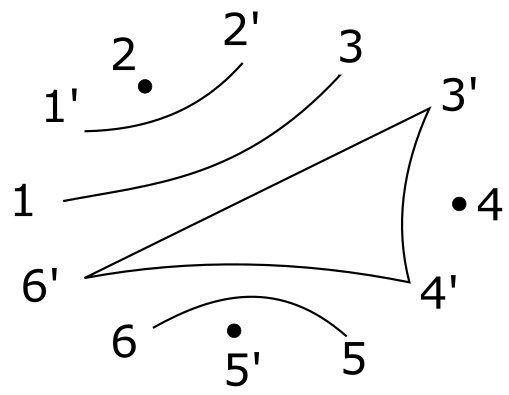}
\caption{The partition $P = \{\{1,3\},\{2\},\{4\},\{5,6\}\}$ is drawn on $\{1, 2, \ldots, 6\}$ and $\krew(P) = \{\{1,2\}, \{3, 4,6\}, \{5\}\}$ is drawn on the primed vertices}
\label{krew}
\end{figure}

By Lemma 3.2 in \cite{RatCat} we can recover the laser set of an $a,b$ noncrossing partition, where $a < b$, from its Kreweras complement. We have a similar result when we generalize to any coprime $a$ and $b$:

\begin{lemma}\label{lasers}
Let $a$ and $b$ be coprime and $(P,Q) = \pi(D)$ have corresponding Dyck path $D$.  If $\krew(P)$ is the Kreweras complement of $P$ then the laser set $\ell(D)$ is given by

\begin{align*}
 \ell(D) &= \{(i, \max(B)) | B \in \krew(P), i \in B, i \neq \max(B)\} \\
& \quad \cup \{ (\max(B), \max(B)) | B \in Q, \rank(B) \neq 0\}
\end{align*}

\end{lemma}
\begin{proof}
The first set consists of all lasers which determine blocks of $P$.  The second set contains those additional lasers, unique to the the case $a > b$, which define nontrivial blocks of $Q$ but not $P$, which are always of the form $(p,p)$ where $p = \max(B)$ for some nontrivial block $B \in Q$. 
\end{proof}

\begin{lemma}\label{krewCom}
If $(P,Q) = \pi(D)$ for an $a,b$ Dyck path $D$ then, when viewed as unlabeled partitions, we have $Q = \krew(P)$.
\end{lemma}
\begin{proof}
First suppose that $i$ and $j$ are in the same block $B$ of $\krew(P)$ where $i \neq j$.  If neither $i$ nor $j$ is equal to $\max(B)$ then by Lemma \ref{lasers} we must have that $(i,\max(B))$ and $(j,\max(B))$ are both lasers in $D$.  Since $\ell(i)$ and $\ell(j)$ hit the same east step, $i$ and $j$ are in the same block of $Q$.  Now suppose $j = \max(B)$.  Then $(i,j)$ is a laser in $D$.  Similarly, if $i = \max(B)$ then $(j,i) \in \ell(D)$. In all cases, $i$ and $j$ are in the same block of $Q$. 

Conversely, suppose $i$ and $j$ are in the same block of $Q$.  Let $B_i$ denote the block in $\krew(P)$ containing $i$ and $B_j$ denote the block in $\krew(P)$ containing $j$.  If $\ell(i)$ and $\ell(j)$ hit the same step immediately following a $Q$ rise above label $k$ then $(i,k)$ and $(j,k)$ are both lasers in $\ell(D)$ with $i \neq j \neq k$.  By the characterization of the laser set given in Lemma \ref{lasers}, we must have that $k = \max(B_i)$ and $k = \max(B_j)$, so $B_i = B_j$.  If $(i,j) \in \ell(D)$ then $j = \max(B_i)$.  If $(j,i) \in \ell(D)$ then $i = \max(B_j)$.  In all cases, $i$ and $j$ are in the same block of $\krew(P)$.  Thus, $Q = \krew(P)$. 
\end{proof}

\begin{proposition}
Given a Dyck path $D$, if $\pi(D) = (P, Q)$ then $Q$ is a noncrossing partition. 
\end{proposition}
\begin{proof}
By the definition of Kreweras complement, $Q = \krew(P)$ is noncrossing.
\end{proof}

We say that two noncrossing partitions $P_1$ and $P_2$ of $\{1, 2, \ldots, n\}$ are \emph{mutually noncrossing} if there do not exist $a < b < c < d$ such that $a$ and $c$ are in the same block of $P_i$ and $b$ and $d$ are in the same block of $P_j$ for $i, j \in \{1, 2\}$ and $i \neq j$.  Equivalently, draw the numbers 1 through $n$ on the boundary of a disk. Then $P_1$ and $P_2$ are mutually noncrossing if when we draw the boundary of the convex hulls of the blocks of $P_1$ with solid lines and the convex hulls of the blocks of $P_2$ in dashed lines, no solid line crosses the interior of a dashed line.  Note that solid-dashed intersections at vertices are permissible.  For example, the picture on the left of Figure \ref{mutNonCross} contains two noncrossing partitions.  One whose blocks are indicated by solid lines, the other whose blocks are indicated by dashed lines.  We see that there are intersections only at labels. On the other hand, the picture on the right in Figure \ref{mutNonCross} shows that if we superimpose a rotated version of the dashed line partition onto the solid line partition, then the partitions are no longer mutually noncrossing. In particular, the $\{1,4\}$ block of dashed line partition crosses both the $\{2,6\}$ and $\{3,5\}$ blocks of the solid line partition. 

\begin{figure}[h]
\includegraphics[scale=.6]{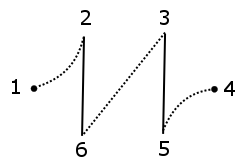} \hspace{1cm}
\includegraphics[scale=.6]{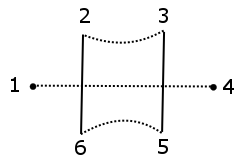}
\hspace{1cm}
\caption{The pair on the left is mutually noncrossing.  The pair on the right is not.}
\label{mutNonCross}
\end{figure}

\begin{proposition}\label{muNonCrossing}
Given a Dyck path $D$, if $\pi(D) = (P, Q)$ then $P$ and $Q$ are mutually noncrossing.
\end{proposition}
\begin{proof}
By definition, $Q = \krew(P)$ has no crossings with $P$. 
\end{proof}

It now makes sense to define the set $NC(a,b)$ of $(a,b)$ noncrossing partitions by
\[ NC(a,b) = \{ \pi(D) | D \mbox{ is an $a,b$-Dyck path}\}.\]  

\section{Rotation, Rank Sequences, and Reflection}

\subsection{The Rotation Operator}
Next, we will define a rotation operator $\rot'$ on $a,b$-Dyck paths that commutes with $\pi$.  In other words, if $\pi(D) = (P, Q)$, then $\pi(\rot'(D)) = \rot^{-1}(\pi(D))$ where $\rot$ is the map acting componentwise on $P$ and $Q$ sending $i$ to $i + 1$, modulo $b-1$, which preserves ranks. 

\begin{defn}
 Let $D = N^{i_1} E^{j_1} \cdots N^{i_m}E^{j_m}$ be the decomposition of $D$ into nonempty vertical and horizontal runs. We define the \emph{rotation operator} $\rot'$ as follows:

\begin{enumerate}
\item If $m = 1$, so that $D = N^aE^b$, we set
\[ \rot'(D) = N^aE^b = D.\]

\item If $m, j_1 > 1$, we set
\[ \rot'(D) = N^{i_1}E^{j_1 - 1} N^{i_2}E^{j_2} \cdots N^{i_m}E^{j_m+1}.\]

\item If $m > 1$ and $j_1 = 1$, let $P = (1, i_1)$ be the westernmost valley of $D$.  The laser $\ell(P)$ fired from $P$ hits $D$ on a horizontal run $E^{j_k}$ for some $2 < k < m$.  Suppose that $\ell(P)$ hits the horizontal run $E^{j_k}$ on step $r$, where $1 \leq r \leq j_k$.  There are two cases to consider: 

If $r = 1$, we set 
\[ \rot'(D) = N^{i_2}E^{j_2} \cdots N^{i_{k-1}}E^{j_{k-1}} N^{i_1}E^{j_k}N^{i_{k+1}}E^{j_{k+1}} \cdots N^{i_m}E^{j_m}N^{i_k}E.\]

If $r > 1$, we set 
\[ \rot'(D) = N^{i_2}E^{j_2} \cdots  N^{i_k}E^{r-1}N^{i_1}E^{j_k - r + 1}N^{i_{k+1}}E^{j_{k+1}} \cdots N^{i_m}E^{j_m + 1}.\]
\end{enumerate}
\end{defn}

This definition is consistent with, but more general than, the one given in Section 3.1 \cite{RatCat}.  The $r=1$ case in $(3)$ will never occur if $a<b$ but can if $a > b$, so this new definition is necessary.  The next proposition shows that $\rot'$ is the path analog of $\rot^{-1}$ on set partitions.  

\begin{proposition}\label{rotPreserve}
The operator $\rot'$ defined above gives a well-defined operator on the set of $a,b$-Dyck paths.  Furthermore, for any Dyck path $D$, if $\pi(D) = (P,Q)$, then $\pi(\rot'(D)) = \rot^{-1}(\pi(D))$.
\end{proposition}

\begin{proof}
First we must check that for any $a,b$-Dyck path $D$, $\rot(D)$ does in fact stay above the line $y = \frac{a}{b}x$.  The definition of this rotation operator differs from the one given in Section 3.1 of \cite{RatCat} for $a < b$ only in the first case in $(3)$, so that is the only case we need to consider here.  It is easiest to explain what happens visually.  In Figure \ref{rot} we break the generic Dyck path at the diagonal slashes into 5 pieces.  The segment labeled 1 is the initial vertical run.  Segment 2 is the single east step which follows.  Segment 3 is the portion of the path between segment 2 and the $Q$-rise preceeding the east step hit by $\ell(1)$. Segment 4 is the aforementioned $Q$-rise.  Segment 5 is the remainder of the path.  The labeled path on the right shows how the inverse rotation operator shifts these segments. 

Since segment 3 stays above a laser fired in $D$, segment 3 in $\rot'(D)$ must stay above the line $y = \frac{a}{b}x$.  Since the segment 4 is a $Q$-rise in $D$, we know that the segment 4 of $\rot'(D)$ has length at most $\lfloor a/b \rfloor$, so the segments 4 and 2 of $\rot'(D)$ stay above the line $y = \frac{a}{b}x$. Since segment 5 stays above the line in $D$, it is clear that it stays above the line in $\rot'(D)$ as well.  Finally, since segment 1 is a single vertical run, it cannot cross the line.  Thus, the path $\rot'(D)$ stays above the line $y = \frac{a}{b}x$ so it is a valid Dyck path.
\begin{figure}[h]
\includegraphics[scale=.5]{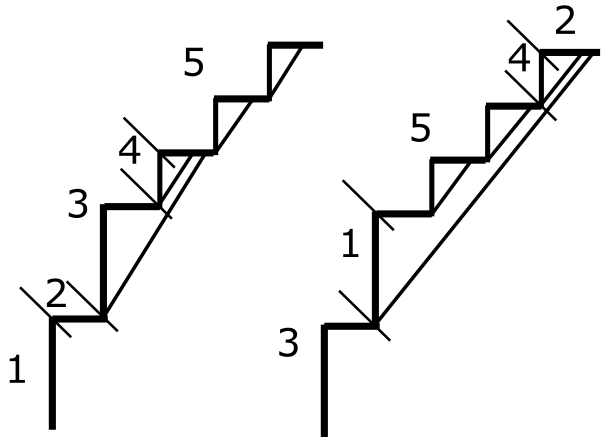}
\caption{The Dyck path on the right is the rotated version of the path on the left}
\label{rot}
\end{figure}
Next we need to argue that $\pi(\rot'(D)) = \rot^{-1}(\pi(D))$. To do this, we simply consider how the lasers change from $D$ to $\rot'(D)$.
\begin{enumerate}
\item The lasers fired from points in segment 5 of $D$ are identical to the lasers fired in segment 5 of $\rot'(D)$, shifted one unit west.
\item The lasers fired within segment 3 which hit just west of a label $s$ in $D$ hit just left of the label $s-1$ in $\rot'(D)$. 
\item The laser from the point labeled 1 in $D$ is replaced by the laser fired from the end of segment 3 in $\rot'(D)$, so the rotated block includes $b-1$ instead in the rotation as desired.
\item Let $t$ be the label at the base of segment 4 in $D$.  Then $t$ and $1$ are in the same block of $Q$ in $\pi(D)$.  In $\rot'(D)$, this laser is fired from $t-1$, and as described in $(3)$ it hits the terminal east step.  Since segment 4 is translated to be the vertical run immediately preceding the terminal east step, the laser fired from $b-1$ in $\rot'(D)$ also hits the terminal east step, so $t-1$ and $b-1$ are in the same block of $Q$ in $\pi(\rot'(D))$, completing the proof that the blocks of $\pi(\rot'(D))$ rotate as desired.
\end{enumerate}
\end{proof}

It now makes sense to define $\rot(D) = \rot'^{-1}(D)$. In other words, $\rot(D)$ is such that $\pi(\rot(D)) = (\rot(P), \rot(Q))$. 

Given an $a,b$-Dyck path $D$, one can obtain a $b,a$-Dyck path $\tau(D)$ by applying the transposition operator $\tau$ which reflects a path about the line $y = -x$, then shifts it such that its southern-most point is at the origin. One might hope that transposition would commute with rotation in the sense that $\tau(\rot(D)) = \rot(\tau(D))$; however, this is not the case, which can be seen immediately from an example.  Let $D = NNNNENENNE$. If we first transpose, we obtain the path $NEENENEEEE$ which corresponds to the partition $A = \{\{1,2\}, \{3,6\}, \{4,5\}\}$.  However, if we first rotate $D$, then transpose, we obtain the partition $B = \{\{1,6\}, \{2,3\}, \{4,5\}\}$, which is not obtainable from $A$ via any rotation.  Since the relevant information of a noncrossing partition is read off from the vertical runs of its associated Dyck path rather than the horizontal runs, which are not preserved under rotation, this is not surprising.  

\subsection{Rank Sequences}

Let $D$ be a Dyck path such that $\pi(D)$ is the labeled pair of noncrossing partitions $(P,Q)$.  If $B$ is a block of $P$, we define $\rank^D_P(B)$ to be the length of the vertical run preceding $\min(B)$ in $D$.  If $B$ is a block of $Q$, we define $\rank^D_Q(B)$ to be the length of the vertical run above $\max(B)$ in $D$.   Since the underlying Dyck path $D$ is almost always clear from context, we will often simply write $\rank_P(B)$ and $\rank_Q(B)$.  Given an $a,b$-Dyck path $D$ such that $\pi(D) = (P,Q) \in NC(a,b)$, we define the associated \emph{$P$ and $Q$ rank sequences}, denoted $S_P$ and $S_Q$ as follows:
\[ S_P := (p_1, p_2, \ldots, p_{b-1})\]
where
\[ p_i = \begin{cases} \rank_P(B)  \mbox{ if } i = \min(B) \mbox{ for some $B \in P$} \\ 0  \mbox{ otherwise. } \end{cases} \]

\[ S_Q := (q_1, q_2, \ldots, q_{b-1})\]
where
\[q_i = \begin{cases} \rank_Q(B)  \mbox{ if } i = \max(B) \mbox{ for some $B \in Q$} \\ 0  \mbox{ otherwise. } \end{cases} \]
To solidify the connection to Dyck paths, observe that given $(P,Q) \in NC(a,b)$ we have $\pi^{-1}(P,Q) = D$ where 
\[D = N^{p_1}EN^{\max(p_2,q_1)}E \cdots N^{\max(p_{b-1}, q_{b-2})}EN^{q_{b-1}}E.\]
More generally, we will simply define the \emph{rank sequence} of $(P,Q)$ to be the sequence given by
\[ R(P,Q) := (p_1, \max(p_2,q_1), \cdots, \max(p_{b-1},q_{b-2}), q_{b-1}).\]
This is precisely the sequence of vertical run lenghts of the Dyck path which gives rise to $(P,Q)$. For example, consider the path and corresponding partitions shown in Figure \ref{nonC}:

\begin{figure}[h]
\includegraphics[scale=.3]{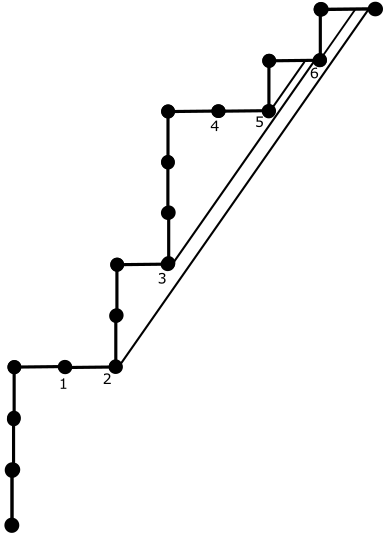}
\includegraphics[scale=.4]{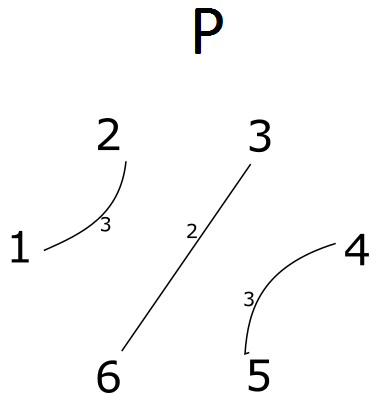}
\hspace{1cm}
\includegraphics[scale=.4]{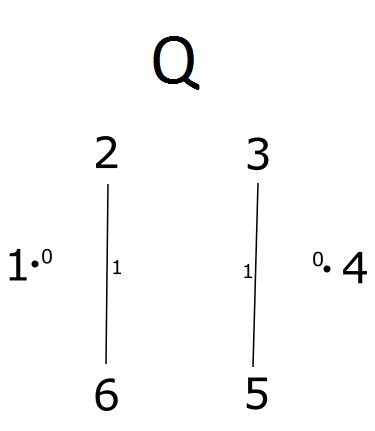}
\caption{A 10,7-Dyck path with corresponding labeled partitions}
\label{nonC}
\end{figure}

We have $S_P = (3,0,2,3,0,0)$, $S_Q = (0,0,0,0,1,1)$, and $R(P,Q) = (3, 0, 2, 3, 0,1,1)$.

\begin{proposition}\label{rankRot}
Let $a$ and $b$ be coprime, $D$ be an $a,b$-Dyck path, and $\pi(D) = (P,Q) \in NC(a,b)$. If $B$ is a block of $P$, then
\[ \rank^D_P(B) = \rank^{\rot(D)}_{\rot(P)}(\rot(B)).\]
If $B$ is a block of $Q$, then
\[ \rank^D_Q(B) = \rank^{\rot(D)}_{\rot(Q)}(\rot(B)).\]
\end{proposition}
\begin{proof}
It will suffice to consider instead the inverse rotation operator $\rot'$ defined for $a,b$-Dyck paths.  This operator preserves vertical run lengths and the underlying block structure of both $P$ and $Q$.  Preservation of rank is clear unless $B$ contains 1, since $\rot'$ just subtracts 1 from every index modulo $b-1$.  If $B$ is in $P$ and contains 1, then by definition of $\rot'$, we translate the entire initial vertical run sequence so it immediately precedes the next element in $B$, after $\rot'$ is applied, so the rank is preserved.  If $B$ is in $Q$ and contains 1, then the $Q$-rise preceding the maximal element in $B$ is translated to the vertical run preceding the terminal east step in the path.  Thus, the $\rank^D_Q(B) = \rank^{\rot(D)}_{\rot(Q)}(B')$ where $B'$ is the block in $\rot(Q)$ coming from $\rot(D)$ which contains $b-1$.  By Proposition \ref{rotPreserve}, we have $B = B'$, so the rank is again preserved.
\end{proof}

Now we show how block ranks respect cardinality under the operation of merging blocks.  In the case where $a < b$, there are no nontrivial $Q$ blocks and merging $P$ blocks of $(P,Q) \in NC(a,b)$ always yields another $a,b$-noncrossing partition.  When $a > b$, the merging of blocks of $P$ results in the splitting of blocks of $Q$, and we need to be careful about how we assign ranks to these split blocks.  This is made precise in the following proposition.  An example follows the end of the proof, which will help clarify the merging operation defined below.

\begin{lemma}\label{pMerge}
Let $a$ and $b$ be coprime positive integers and $D$ be an $a,b$-Dyck path such that $\pi(D) = (P,Q) \in NC(a,b)$, and $B$ and $B'$ be two blocks of $P$.   Let $P'$ be the result of replacing $B$ and $B'$ in $P$ by $B \cup B'$.  If $P'$ is a noncrossing partition, then $(P', Q') \in NC(a,b)$ where $Q' = \krew(P')$ and $\rank_{P'}(B \cup B') = \rank_P(B) + \rank_P(B')$.  For any block $C' \in Q'$, if $\max(C') = \max(C)$ for some $C \in Q$, then $\rank_{Q'}(C') = \rank_Q(C)$.  Otherwise $\rank_{Q'}(C') = 0$. 
\end{lemma}
\begin{proof}
Without loss of generality assume $\min(B) < \min(B')$. The Dyck path operation which merges $B$ and $B'$ consists of removing the vertical run of length $\rank(B')$ atop $\min(B') - 1$ and adding $\rank(B')$ north steps to the vertical run atop $\min(B) - 1$.  We will now verify that this indeed gives the desired result.  Let $D'$ denote the Dyck path which results from applying this operation to $D$. The only lasers $\ell(p)$ which are potentially affected by this operation are those such that $\min(B) - 1 \leq p \leq \min(B') - 1$.  For now, assume $p \neq \min(B) - 1$.  If $\ell(p)$ hits west of $\min(B') - 1$ in $D$ then it is unchanged in $D'$, so we need only consider the case where it hits east of $\min(B') - 1$.  Observe that the horizontal distance from $\ell(\min(B') - 1)$ and $\ell(P)$ is at most 1.  To see this, suppose it were greater than 1.  Then there would exist a label $q > \max(B')$ such that $\ell(\min(B') - 1)$ hits $D$ west of $q$ and $\ell(p)$ hits $D$ east of $q$.  Let $B''$ be the block containing $q$.  Then $\min(B') \leq \min(B'') < \max(B') < q \leq \max(B'')$ which contradicts the fact that $P$ is noncrossing. This implies that in $D'$, all such lasers hit the east step hit by $\ell(\min(B')-1)$.  Each of these lasers is translated vertically by $\rank(B')$ units, so the block structure and ranks of other blocks of $P$ remain unchanged.

Now consider the case where $p = \min(B) - 1$.  If $\ell(p)$ hits $D$ east of $\ell(\min(B') - 1)$ then $\ell(p)$ is the same laser in $D$ and $D'$.  Since $\ell(\min(B') - 1)$ disappears, all labels of $B'$ become visible to labels of $B$, so the blocks union and the ranks sum, as desired. Now suppose $\ell(p)$ hits $D$ west of $\ell(\min(B')-1)$.  Let $C$ denote the block containing $\min(B') - 1$.  Then we must have $\min(C) \leq \min(B) - 1 < \min(B') - 1$, which implies that merging $B$ and $B'$ would create a crossing, a contradiction.  

By Lemma \ref{krewCom} we must have $Q' = \krew(P')$. Let $C' \in Q'$ and suppose $\max(C') = \max(C)$ for some $C \in Q$.  Since the merge operation preserves all vertical run lengths except two, each of which is a $P$-rise, we know that the rank of $C$ must be preserved.  On the other hand, the merge operation removes some lasers from $D$, so some elements which were originally in $C$ will no longer fire lasers, forcing them to be in their own block of rank 0 in $Q'$. 
\end{proof}

For example, consider once again the 10,7-Dyck path from Figure \ref{nonC}, along with its associated noncrossing partitions $P$ and $Q$.  Suppose we would like to merge the blocks $B = \{1,2\}$ and $B' = \{3,6\}$ in $P$.  Doing so gives the partition $P' = \{\{1,2,3,6\}, \{4,5\}\}$ which is indeed noncrossing, so $(P', Q') \in NC(a,b)$.  We have $Q' = \krew(P')$.  Since 6 was the maximal element of the $Q$ block $\{2,6\}$ with rank 1, the $\rank_{Q'}(\{6\}) = 1$.  Since 2 was not a maximal element of a block in $Q$, its rank is now 0. All other blocks and ranks are preserved.

Next let's examine how we need to modify $D$ to obtain $D'$, where $\pi(D') = (P', Q')$.  We remove the vertical run of length $2 = \rank(\{3,6\})$ above 2 and adding north steps to the vertical run atop $0 = \min(\{1,2\}) - 1$.  The new path $D'$, along with $P'$ and $Q'$, each with rank labels shown, are shown below in Figure \ref{merge}. 

\begin{figure}[h]
\includegraphics[scale=.4]{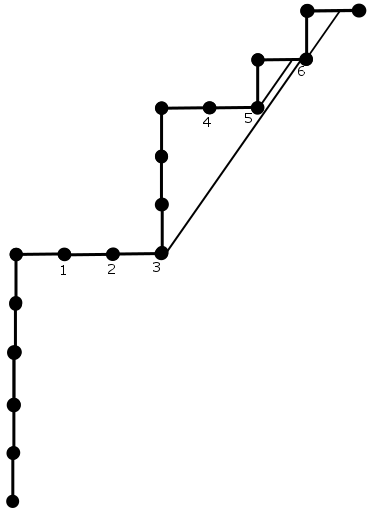}
\includegraphics[scale=.8]{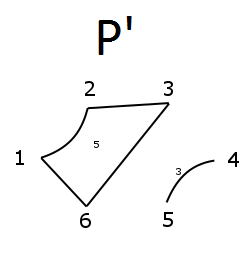}
\hspace{1cm}
\includegraphics[scale=.8]{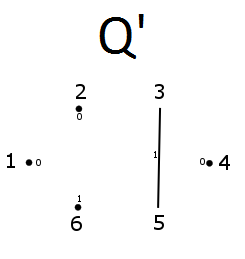}
\caption{The Dyck path $D'$, along with $P'$ and $Q'$}
\label{merge}
\end{figure}

We now discuss the problem of determining whether an arbitrary labeled pair of noncrossing partitions is in fact a member of $NC(a,b)$.  First, we will define a partial order $\preceq$ on the blocks of any pair $(P,Q)$ of noncrossing partitions by

\[ B' \preceq B \mbox{ if } \begin{cases} B, B' \in P \mbox{ and } [\min(B'), \max(B')] \subset [\min(B), \max(B)] \\ B' \in Q, B \in P, \mbox{ and } \max(B') \in [\min(B), \max(B)] \\ B' = B, B \in Q \end{cases}\]

This partial order will tell us when we can absorb the rank of a block of $Q$ into the rank of a block of $P$ to obtain a new element of $NC(a,b)$.  

\begin{lemma}\label{qMerge}
Let $(P,Q) \in NC(a,b)$, $B \in P$, $B' \in Q$, and suppose $B'$ is covered by $B$ under $\preceq$.  Define a pair $(P', Q')$ as follows:  $P'$ is obtained from $P$ by simply increasing the rank of $B$ by $\rank(B')$.  $Q' = \krew(P')$ and $B'$ is assigned rank 0.  Then $(P', Q') \in NC(a,b)$. 
\end{lemma}
\begin{proof}
Let $D$ denote the Dyck path such that $\pi(D) = (P,Q)$.   The Dyck path operation which performs the desired merge moves the $Q$-rise from above $\max(B')$ to the $P$-rise above $\min(B) - 1$.  This clearly increases the rank of $B$ by $\rank(B')$. Furthermore, the east step hit by $\ell(p)$ for all $p \in B'$ is preceded by a $Q$-rise of length 0, so the block rank becomes 0. To see that all other blocks and ranks are fixed by this process it is enough to consider how the lasers are affected.  The laser $\ell(p)$ is translated vertically by $\rank(B')$ units if and only if $\min(B) \leq p \leq \max(B')$.  However, the portion of $D$ which lies between these labels is also translated vertically by $\rank(B')$, so no changes can take place unless $\ell(p)$ hits $D$ east of $\max(B')$. 

Without loss of generality, assume $p$ is the largest label such that $\ell(p)$ hits east of $\max(B')$, and suppose that $\ell(\max(B'))$ and $\ell(p)$ fail to hit the same east step.  Then there must exist a label $q$ which lies between $\ell(\max(B'))$ and $\ell(p)$.  Let $C$ be the block of $P$ containg $q$.  Then 
\[ \min(B) \leq p < \min(C) \leq \max(B') \leq \max(C) \leq \max(B).\]
where the last inequality follows since $P$ is noncrossing. Thus, $B' \prec C \prec B$, contradicting the fact that $B$  covers $B'$.  Thus, we may assume that $\ell(\max(B'))$ and $\ell(p)$ hit the same east step.  We know that $\ell(\max(B'))$ hits the east step immediately following the label $\max(B')$, so $\ell(p)$ must as well.  In the modified Dyck path, this step is translated down so the lasers will so the points where they make contact with the east step will shift west.  Since the westernmost laser which hits this east step is $\ell(\max(B'))$, no laser will hit further west than the point labeled $\max(B')$.  Thus, $\ell(p)$ will still hit the same east step it originally did.  Thus, the block structure is preserved and $B'$ now has rank 0.
\end{proof}

\begin{figure}[h]
\includegraphics[scale=.4]{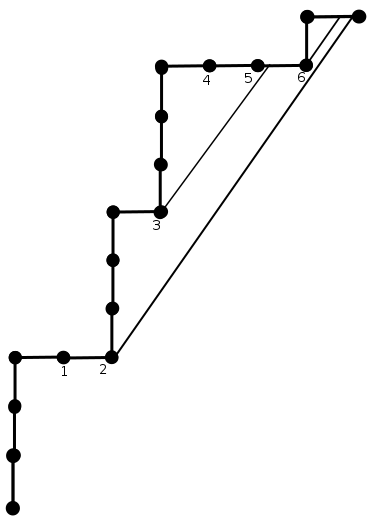}
\includegraphics[scale=.7]{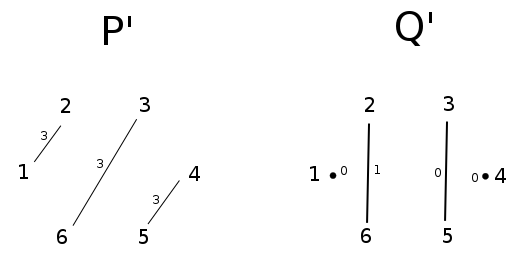}
\caption{The path and partitions obtained from the merge}
\label{Qmerge}
\end{figure}

For example, consider again the 10,7-Dyck path from Figure \ref{nonC}.  Suppose we wish to merge the $Q$-block $\{3,5\}$ of rank 1 with the $P$-block $\{3,6\}$ of rank 2.  The Dyck path operation removes the $Q$-rise of length 1 from above 5 and places it into the vertical run above 2.  On the level of partitions, we obtain $P'$ by increasing the rank of $\{3,6\}$ by 1.  We obtain $Q'$ by changing the rank of $\{3,5\}$ to 0.  Figure \ref{Qmerge} shows the resulting path and partitions $P'$ and $Q'$.

Unlike in the case where $a < b$, ranks are no longer uniquely determined by the partition structure.   However, slope considerations do limit which ranks can possibly be assigned to a given block.  

\begin{defn}
Let $a$ and $b$ be coprime positive integers and $(P,Q) \in NC(a,b)$.  We say that a block $B$ of $P$ satiesfies the \emph{rank condition} if

\[  (\max(B) - \min(B) + 1)\frac{a}{b}  \leq \sum_{B' \preceq B} \rank(B') \leq (\max(B) - \min(B) + 1)\frac{a}{b} + \frac{a}{b} .\]
\end{defn}

Note that here we use rank to indicate the label of the block $B'$ rather than anything having to do with vertical run lengths.  This way, it makes sense to ask whether any block $B \in P$ for a labeled pair $(P,Q)$ satisfies the rank condition.  Such an inequality must hold for $(P,Q) \in NC(a,b)$.  The argument is essentially identical to that given in Proposition 3.8 of \cite{RatCat} when one considers the fact that $Q$ block ranks also contribute vertical runs. When $a < b$ the lower and upper bounds necessarily agree and uniquely determine the rank of each block, which is why labels on the partition were unnecessary in that case.

The following theorem characterizes precisely when a pair $(P,Q)$ belongs to $NC(a,b)$.  This is a generalization of Theorem 3.15 in \cite{RatCat}, which provides such a characterization when $a < b$. 

\begin{theorem}\label{char}
Let $(P,Q)$ be a pair of labeled mutually noncrossing partitions and $a$ and $b$ be fixed, coprime positive integers.  Then $(P,Q) \in NC(a,b)$  if and only if the following conditions hold:
\begin{enumerate}
\item $\sum_{B \in P} \rank(B) + \sum_{B' \in Q} \rank(B') = a$
\item We have $\rank(B) < a/b$ for all $B \in Q$
\item $Q = \krew(P)$
\item The rank condition holds for all blocks in $\rot^m(P,Q)$ for $1 \leq m \leq b-1$.
\end{enumerate}
\end{theorem}
\begin{proof}
First suppose that $(P,Q) \in NC(a,b)$.  Then there exists a Dyck path $D$ such that $\pi(D) = (P,Q)$ and the vertical sequence of $D$ comes from the ranks of blocks in $P$ and $Q$, so they must sum to $a$.  The second condition follows immediately from slope considerations.  By Lemma \ref{krewCom}, the Kreweras complement uniquely determines $Q$.  Finally, Proposition \ref{rankRot} implies that condition (4) must hold. 

Now suppose that we're given a pair $(P,Q)$ which satisfies each of the conditions $(1) - (4)$. In the case where $a < b$, Proposition \ref{char} reduces to Proposition 3.5 of \cite{RatCat}, so we will only consider the case $a > b$ here.   Although we previously defined rank sequences only for $(P,Q) \in NC(a,b)$, it makes sense to think of them for any labeled pair of noncrossing partitions.  Let $D_{(P,Q)}= N^{p_1}EN^{\max(p_2,q_1)}E \cdots N^{\max(p_{b-1}, q_{b-2})}EN^{q_{b-1}}E$.  By condition (1), $D_{(P,Q)}$ will actually have height $a$ so it is indeed an $a,b$-lattice path.  By Lemma \ref{lasers} we can immediately read off from $\krew(P)$ what the laser set of $D$ must be in order to have $\pi(D) = (P,Q)$.  

For example, consider the pair $P = \{\{1,3\}, \{2\}\}$ with ranks 5 and 1 respectively, and $Q = \{\{1,2\}, \{3\}\}$ with ranks 1 and 0 respectively.  The pair $(P,Q)$ is not in $NC(7,4)$.  It satisfies conditions (1) - (3), and each block of $P$ satisfies the rank condition.  By Lemma \ref{lasers}, $L(D) = \{(1,2), (2,2)\}$.  Now let's examine the rank sequences of $(P,Q)$.  We have $S_Q = (5,1,0)$ and $S_Q = (0,1,0)$.   Thus, the Dyck path which would have to give rise to $(P,Q)$ (if such a Dyck path exists) must look like $D_{(P,Q)} = NNNNNENENEE$, shown in Figure \ref{charPath}. 

\begin{figure}[h]
\includegraphics[scale=.4]{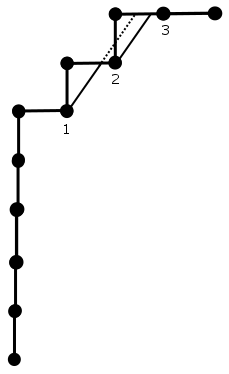}
\caption{The candidate Dyck path for the pair $(P,Q)$.}
\label{charPath}
\end{figure}
 We immediately see that the laser set is $\{(1,1), (2,2)\}$.  However, if we extend $\ell(1)$ further, along the dashed line shown in Figure \ref{charPath}, we see that it was certainly set up to hit in the appropriate spot, if only another part of $D$ had not gotten in the way.   More generally, this is true of any laser fired from the bottom of a $P$-rise since each block of $P$ satisfies the rank condition.  Moreover, this is the only thing which can go wrong since a laser fired from the bottom of a $Q$-rise will always hit the next east step by condition (2). 

This is a special feature of the case $a > b$.  Namely, the interval
\[ \left[ \left\lceil m\frac{a}{b} \right\rceil, \left\lfloor (m+1)\frac{a}{b} \right\rfloor \right]\]
is nonempty for any $m \in [b-1]$, so the rank condition is just another way of saying that the laser which cuts out a block can and will hit on the appropriate east step.  However, the proof we give here for $a > b$ differs from the one given in \cite{RatCat} of Proposition 3.5 since knowing $p$ and which east step $\ell(p)$ hits no longer uniquely determines the height difference of $p$ and that east step.

Let $i$ and $k$ be such that $(i,k-1)$ should be a laser according to $\krew(P)$, but such that it fails to be a laser in $D = D_{(P,Q)}$ because it first hits a horizontal segment of $D$ whose easternmost endpoint is labeled $j$.  To simplify things, we can repeatedly apply Lemmas \ref{pMerge} and \ref{qMerge} to $(P,Q)$ to obtain a pair $(P', Q')$ which has the same problem when we consider $D' = D_{(P', Q')}$.   Since $NC(a,b)$ is closed under the merge operations of these lemmas, it will suffice to derive a contradiction for $D'$ coming from $(P', Q')$ of the form shown in Figure \ref{charProof}.  In particular, $D'$ will contain exactly three vertical runs: the initial vertical run above the origin of length $A$, the vertical run atop $i$ of length B, and the vertical run atop $j$ of length $C$.

\begin{figure}[h]
\includegraphics[scale=.7]{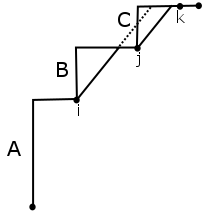}
\caption{A simplified candidate Dyck path $D'$}
\label{charProof}
\end{figure}

Since each laser is of slope $\frac{a}{b}$, we must have 
\[ (j-i)a/b > B. \hspace{1cm}(*) \]
Furthermore, $\ell(j)$ must hit $D'$ on the east step between the labels $k-1$ and $k$.  To see this, observe that if $\ell(j)$ hit west of $k-1$ then there would be no interference with $\ell(i)$. If $\ell(j)$ hit east of $k$ then we would have one block containing $j$ and $k+1$, and another block containing $i$ and $k$, which would imply a crossing.  There are now two cases to consider: either $j$ is at the bottom of a $P$-rise or a $Q$-rise.  

First, suppose $j$ is at the bottom of a $P$-rise.  Then we have $P' = \{B_1, B_2, B_3\}$ where $B_1 = [1,i] \cup [k,b-1]$, $B_2 = [i+1, j]$ and $B_3 = [j+1, k-1]$, and $Q'$ consists of only trivial blocks of rank 0.  By condition (4), every rotation of $(P',Q')$ satisfies the rank condition for each rotation block of $P'$, so we may as well assume $P'$ is rotated so that $B_1 = [k-i, b-1]$, $B_2 = [1, j-i]$, and $B_3 = [j+1-i, k-1-i]$. Since $B_2$ satisfies the rank condition, we must have that $ (j-i)a/b  \leq \rank(B_2) = B$.  However, by $(*)$ we have $(j-i)a/b > B$, a contradiction. 

Now suppose $j$ is at the bottom of a $Q$-rise.  Then we have $k = j+1$ and $P' = \{B_1, B_2\}$ where $B_1 = [1,i] \cup [j+1, b-1]$ and $B_2  = [i+1, j]$.  The partition $Q'$ consists of a single nontrivial block $B' = \{i,j\}$ of rank $C$.  As before, it will suffice to consider rotations of $(P', Q')$, so we may now assume that $B_1 = [j+1-i, b-1]$, $B_2 = [1, j-i]$ and $B' = \{j-i, b-1\}$.  Since it is no longer the case that $B' \preceq B_2$ and $B_2$ satisfies the rank condition, we must have $ (j-i)a/b  \leq \rank(B_2) = B$.  However, this contradicts $(*)$ which guarantees $B < (j-i)a/b$. In either case, we conclude that if $(P,Q)$ satisfies conditions (1) through (4) then $(P,Q) \in NC(a,b)$. 
\end{proof}

\subsection{Reflection}

It was shown in \cite{RatCat} that $NC(a,b)$ is closed under the \emph{reflection} operator, given by the permutation 
\[\rfn = \left(\begin{array}{ccccc} 1 & 2 & \cdots & b-2 & b-1 \\ b-1 & b-2 & \cdots & 2 & 1\end{array}\right).\]
When $a > b$ we achieve closure under reflection provided that we choose the appropriate reflection operator on $Q$.  Define $\rfn'$ as follows:  

\begin{align*}
\rfn' &= \left(\begin{array}{ccccc} 1 & 2 & \cdots & b-2 & b-1 \\ b-2 & b-3 & \cdots & 1 & b-1\end{array}\right)
\end{align*}

To simplify notation, define a rotation operator $\rfn''$ by 
\[ \rfn''(B) = \begin{cases} \rfn(B) & \mbox{ if } B \in P \\ \rfn'(B) & \mbox{ if } B \in Q \end{cases}\]

\begin{proposition}
Let $a$ and $b$ be coprime.  If $(P,Q) \in NC(a,b)$ then \linebreak[4] $(\rfn(P), \rfn'(Q)) \in NC(a,b)$, where block labels are preserved. 
\end{proposition}
\begin{proof}
Since $(P,Q) \in NC(a,b)$ it must satisfy conditions $(1) - (4)$ in Theorem \ref{char}. Since ranks are preserved, we have 
\[ \sum_{B \in \rfn(P)} \rank(B) + \sum_{B' \in \rfn'(Q)} \rank(B') = \sum_{B \in P} \rank(B) + \sum_{B' \in Q} \rank(B') = a\]
and $\rank(B) < a/b$ for all $B \in Q$.  By the way we have defined $\rfn'$, we have that $\rfn'(Q) = \krew(\rfn(P))$.  Lastly, for any $B \in P$ there exists $m$ such that $B' \preceq B$ in $(P,Q)$ if and only if $\rot^m(\rfn''(B')) \preceq \rot^m(\rfn(B))$ in $(\rfn(P), \rfn'(Q))$.  Thus, every block of $\rfn(P)$ satisfies the rank condition.  By Theorem \ref{char}, we have that $(\rfn(P), \rfn'(Q)) \in NC(a,b)$. 
\end{proof}

\begin{corollary}
Let $a$ and $b$ be coprime.  The set $NC(a,b)$ of $a,b$ noncrossing partitions is closed under the dihedral action $\langle \rot, \rfn'' \rangle$. 
\end{corollary}

\section{$d$-modified Rank Sequences}

We now set out to count the number of $(P,Q) \in NC(a,b)$ which are invariant under $d$-fold rotation, which will ultimately allow us to prove an instance of the cyclic sieving phenomenon.  To do this, we generalize the notion of $d$-modified rank sequences to our pairs $(P,Q)$.  Along the way, we generalize the \emph{good}, \emph{very good}, and \emph{noble} sequences defined in \cite{RatCat}. We will conclude by showing that these $d$-modified rank sequences are in bijective correspondence with those $(P,Q)$ which are invariant under $d$-fold rotation.  This will reduce our problem to counting these sequences.

Let $d|n$ and $P$ be a noncrossing partition of $[n]$ which is invariant under $\rot^d$. Given a block $B$ of $P$, we say $B$ is a \emph{central block} if $\rot^d(B) = B$.  Clearly $P$ can contain at most one central block.  We say $B$ is a \emph{wrapping block} if $B$ is not central and $[\min(B), \max(B)]$ contains every block in the $\langle \rot^d \rangle$-orbit of $B$.  The $\langle \rot^d \rangle$-orbit of a block can contain at most one wrapping block. 

For the remainder of this section, fix positive coprime integers $a$ and $b$, and an integer $1 \leq d < b-1$ such that $d | (b-1)$.  Let $NC_d(a,b)$ denote the set of $(P,Q) \in NC(a,b)$ which are invariant under $\rot^d$.  Given $(P,Q) \in NC_d(a,b)$, we define the $d$-modified $P$ and $Q$ rank sequences as follows: 

\[S_P^d := (p_1, \ldots, p_d) \mbox{ and }  S_Q^d := (q_1, \ldots, q_d)\]

where

\begin{align*}
p_i &:= \begin{cases} \rank_P(B)  &\mbox{if } i = \min(B) \mbox{ for a noncentral, nonwrapping block} \\
& B\in P \\ 
0  \mbox{ otherwise } \end{cases} \\
q_i &:= \begin{cases} \rank_Q(B)  &\mbox{if } b-1 - d + i = \max(B) \mbox{ for a noncentral,} \\ 
& \mbox{nonwrapping block $B \in Q$} \\ 0  \mbox{ otherwise. } \end{cases} 
\end{align*}

It might seem surprising that in the definition of $q_i$ we consider the largest $d$ elements of $[b-1]$ rather than the smallest $d$ elements, as we did for $p_i$.  The reason comes from the fact that $Q$ ranks are defined in terms of maximal block elements rather than minimal ones.  In particular, $\{b-d, b-d+1, \ldots, b-1\}$ is guaranteed to contain at least one maximal element of a nonwrapping $Q$ block in a $\langle \rot^d\rangle$-orbit, whereas $\{1, 2, \ldots, d\}$ might not. 

For example, consider the pair $(P,Q)$  in $NC_3(10,7)$ given in Figure \ref{nonC}.  We have $S_P^3 = (3,0,0)$ since 1 is the minimal element of $\{1,2\}$ which has rank 3, 2 is not the minimal element of a block of $P$, and 3 is the minimal element of a central block of $P$.  We also have $S_Q^3 = (0,1,0)$ since 4 is in a trivial $Q$ block, 5 is the maximal element of a Q block of rank 1, and 6 is the maximal element of a wrapping block of $Q$.  Had we instead only recorded $Q$ ranks of 1, 2, and 3, we would have recorded $(0,0,0)$ and lost all information about the structure of $Q$.

\begin{lemma}\label{dRot}
Let $(P,Q) \in NC_d(a,b)$ and $S_P^d$ and $S_Q^d$ be the $d$-modified $P$ and $Q$ rank sequences of $(P,Q)$.  Then we have
\[ S_P^d(\rot(P,Q)) = \rot(S_P^d(P,Q))\]
and
\[S_Q^d(\rot(P,Q)) = \rot(S_Q^d(P,Q)).\]
\end{lemma}
\begin{proof}
The first equality follows by applying the same argument as in the proof of Lemma 4.2 in \cite{RatCat} to the $d$-modified $P$ rank sequences and using Proposition \ref{rankRot}.  We present here the proof of the second equality.  Let 
\[S_Q^d(\rot(P,Q)) = (q_1', q_2', \ldots, q_d')\]
 be the $d$-modified $Q$ rank sequence of $\rot(P,Q)$ and $1 \leq i \leq d$.  We will show that $q_i' = q_{i-1}$ where subscripts are interpreted modulo $d$. 

\textbf{Case 1:} $2 \leq i \leq d$.  If $q_{i-1} > 0$ then $i-1 = \max(B)$ for some non-central, non-wrapping block $B \in Q$.  Thus $i = \max(\rot(B))$ and $\rot(B)$ is non-central and non-wrapping so $q_i' = q_{i-1}$. Next suppose $q_{i-1} = 0$.  If $i-1$ was not the maximal element of a block of $Q$ then $i$ is not the maximal element of $\rot(Q)$, so $q_i' = 0$.  If $i-1 = \max(B)$ for a wrapping block $B$ then $\rot(B)$ is wrapping and $i = \max(\rot(B))$, so $q_i' =  0$.  If $i-1 = \max(B)$ for a central block $B$ then $\rot(B)$ is central and $i = \max(\rot(B))$ so $q_i'$ is 0. 

\textbf{Case 2:} $i=1$. Suppose $b-1 = \max(B)$ for some non-central, non-wrapping block $B$. By rotational symmetry, $\rot^{b-1-d}(B)$ is a non-central, non-wrapping block of $Q$ with max $b-1-d$ and rank $q_d$.  Thus, $b-d$ is the max of a rotated block in $\rot(Q)$, so we have $q_1' = q_d$.  Now suppose $b-1 \in B$ where $B$ is central. Then $1 \in \rot(B)$ which is also central, so  $q_1' = q_d = 0$.  Lastly, suppose $B$ is central.  If $b-1$ is the only element of $B$ in $\{b-d, b-d+1, \ldots, b-1\}$ then 1 is not the maximal element of $\rot(B)$ so by rotational symmetry, $b-d$ is not the maximal element of a $Q$ block and we have $q_1' = 0$.  On the other hand, if $b-1$ is not the only element of $B$ in $\{b-d, b-d+1, \ldots, b-1\}$ then $\rot(B)$ is still wrapping so $q_1' = 0$.
\end{proof}

Define the set of \emph{good sequence pairs} to be the set of nonnegative integer sequence pairs of length $d$, $(S_P^d, S_Q^d) =((p_1, \ldots, p_d), (q_1, \ldots, q_d))$, such that the following hold:
\begin{itemize}
\item $p_i = 0$ or $p_i > a/b$ for each $i \in [d]$
\item $q_i < a/b$
\item $\sum_{i=1}^d p_i + q_i \leq ad/(b-1)$, and 
\item there does not exist $i \in [d]$ such that both $p_{i+1}$ and $q_i$ are nonzero, where subscripts are interpreted modulo $d$. 
\end{itemize}
When $a < b$, $S_Q^d$ is always a sequence of all 0's and the sequences $S_P^d$ are exactly the good sequences defined in \cite{RatCat}.  Our goal is to show that the set of $\rot^d$-invariant pairs of noncrossing partitions in $NC(a,b)$ are in bijective correspondence with the set of good sequence pairs.  The next few pages will consist of a series of somewhat technical lemmas and propositions that will build up to a proof of this bijection.

We say $(P,Q) \in NC_d(a,b)$ is \emph{noble} if the following conditions hold:
\begin{enumerate}
\item neither $P$ nor $Q$ contains any wrapping blocks
\item if $P$ contains a central block $B$ then $1 \in B$
\item if $Q$ contains a central block $B$ then $b-1 \in Q$.
\end{enumerate}

Observe that since $P$ and $Q$ are mutually noncrossing, there can be at most one central block in total.  

\begin{lemma}\label{centralQb}
Suppose $(P,Q) \in NC_d(a,b)$ and that $Q$ contains a central block $B$.  Then either $b-1 \in B$ or $P$ contains a wrapping block, but not both.
\end{lemma}
\begin{proof}
Let $(P,Q) \in NC_d(a,b)$.  First we will show that at least one of the two things must happen.  Suppose toward a contradiction that $Q$ contains a central block $B$ such that $b-1 \notin B$ and that $P$ contains no wrapping blocks.  Let $D$ be the Dyck path such that $\pi(D) = (P,Q)$.  Figure \ref{centralQ} shows a simplified version of what $D$ might look like.  Then $f = \max(B) + d - (b-1) \in B$ and $\ell(f)$ must hit the east step immediately following $\max(B)$ in $D$.  Since $f$ fires a laser and is not the maximal element of a $Q$ block, the vertical run above $f$ (boldened in Figure \ref{centralQ}) consists of at least $a/b$ north steps.  This implies that $f+1$ is the minimal element of a block of $P$.  By $d$-fold symmetry, this means that $\max(B) + 1$ is the minimal element of a block of $P$.  In particular, $\max(B) + 1 \leq b-1$ and the block containing $\max(B) + 1$ cannot be wrapping.  Thus, $\max(B) + 1$ is immediately preceded by a vertical run (the bold vertical run above $\max(B)$ in Figure \ref{centralQ} of length greater than $a/b$.   On the other hand, $\max(B)$ is the maximal element of a $Q$ block so it must be the bottom of a vertical run of length less than $a/b$, a contradiction. 

\begin{figure}[h]
\includegraphics[scale=.4]{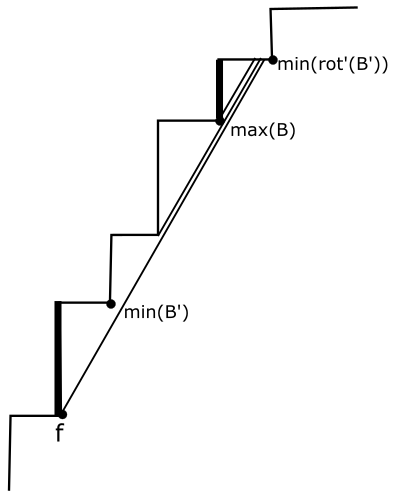}
\caption{}
\label{centralQ}
\end{figure}

Now suppose that $b-1 \in B$ and $P$ contains a wrapping block $A$.  Since $P$ and $Q$ are mutually noncrossing we must have that $b-1 \in A$, and no other element of $\{b-d, \ldots, b-1\}$ is in $A$. By $d$-fold symmetry, $b-1-d$ is the minimal element of the $P$ block $A' = \rot^{-d}(A)$, and also in $B$.  Since $A$ is wrapping, it necessarily contains some element of $\{1, \ldots, d\}$, which implies that $A'$ contains a second smallest element $k$ of $\{b-d, \ldots, b-1\}$.  However, $k \neq b-1$ since $P$ and $Q$ are mutually noncrossing.  Since $b-d-1$ is in $B$, $\ell(b-d-1)$ hits the east step immediately following $b-1$.  This necessarily separates the label $k$ from the label $b-d-1$, contradicting the fact that they are in the same block of $P$.
\end{proof}

\begin{lemma}\label{p1}
Suppose $(P,Q) \in NC_d(a,b)$ and $P$ contains a central block containing 1.  Then $Q$ can contain no wrapping blocks.  
\end{lemma}
\begin{proof}
Let $(P,Q) \in NC_d(a,b)$ where $P$ contains a wrapping block with 1, and suppose toward a contradiction that $Q$ contains a wrapping block $B$. Since $P$ and $Q$ are mutually noncrossing we must have $1 \in B$, but no other element of $\{2, \ldots, d\}$ is in $B$. Let $D$ be the Dyck path such that $\pi(D) = (P,Q)$.  There are two cases to consider:

\textbf{Case 1:} $\max(B) < b-1$

$D$ must contain a vertical run of length less than $a/b$ above $\max(B)$, which implies that $\max(B) + 1$ is not the minimal element of a block of $P$.  By $d$-fold symmetry, this implies that $\max(B) + d - (b-1) + 1$ is not the minimal element of a $P$ block. Let $B' = \rot^d(B)$.  Then $f = \max(B) + d - (b-1)$ is a nonmaximal element of $B'$ so $f$ fires a laser in $D$ and is at the bottom of a vertical run of length at least $a/b$.  This implies $f+1$ is a minimal element of a block of $P$, a contradiction.

\textbf{Case 2:} $\max(B) = b-1$

Since $b-1$ and $1$ are in the same block of $Q$, by $d$-fold symmetry we must have that $d$ and $d+1$ are in the same block of $Q$. Thus $d$ fires a laser in $D$ and is at the bottom of a vertical run of length at least $a/b$.  However, that would imply that $d+1$ is the minimal element of a block of $P$, but by symmetry $d+1$ is in the central block of $P$ whose minimal element is 1, a contradiction. 
\end{proof}

\begin{proposition}
Every $\rot$-orbit in $NC_d(a,b)$ contains at least one noble partition.
\end{proposition}
\begin{proof}
If $P$ contains a central block, rotate it so that it contains 1.  Since $P$ itself is noncrossing, there are no wrapping blocks in $P$.  By Lemma \ref{p1}, there can be no wrapping $Q$ blocks. If $Q$ contains a central block, rotate it so that it contains $b-1$.  As before, since $Q$ is noncrossing, it cannot contain any wrapping blocks.  By Lemma \ref{centralQb}, there can be no wrapping $P$ blocks. Now assume that there is no central block, and suppose that either $P$ or $Q$ contains a wrapping block $B$.  Rotate $(P,Q)$ until the first time $B$ is no longer wrapping.  The result of this rotation cannot introduce any new wrapping blocks, so we have decreased the total number of wrapping blocks by at least 1.  Continue in this way until no wrapping blocks in either $P$ or $Q$ remain. 
\end{proof}

Let $(S_P^d, S_Q^d)$ be a good sequence pair.  Let $s = \sum_{i=1}^d p_i + q_i$ and $c = a - s(b-1)/d$.  We call $(S_P^d, S_Q^d)$ \emph{very good} if 
$c = 0$,  
if $p_1 = 0$ and $c > a/b$, 
or if $q_d = 0$ and $ 0 < c < a/b$. 
Define a map 
\[ L : \{\mbox{very good sequences}\} \to \{ \mbox{lattice paths from (0,0) to } (b,a)\}\]
as follows. If $(S_P^d, S_Q^d)$ is a very good sequence pair, let $L(S_P^d, S_Q^d)$ be determined as follows.  

\textbf{Case 1:} If $c=0$, 
\[\mbox{If $p_1 = 0$, set } L(S_P^d, S_Q^d) = (N^{\max(p_2, q_1)}E \cdots N^{\max(p_d, q_{d-1})}EN^{q_d}E)^{(b-1)/d}E.\]
\[\mbox{If $q_d = 0$, set } L(S_P^d, S_Q^d) = (N^{p_1}EN^{\max(p_2, q_1)}E \cdots N^{\max(p_d, q_{d-1})}E)^{(b-1)/d}E.\]

\textbf{Case 2:} If $c > a/b$, set
\[ L(S_P^d, S_Q^d) = N^{c}E(N^{\max(p_2, q_1)}E \cdots N^{\max(p_d, q_{d-1})}EN^{q_d}E)^{(b-1)/d}\]


\textbf{Case 3:} If $0 < c < a/b$, set
\[ L(S_P^d, S_Q^d) = (N^{p_1}EN^{\max(p_2, q_1)}E \cdots N^{\max(p_d, q_{d-1})}E)^{(b-1)/d}N^{c}E\]


We define a very good sequence pair $(S_P^d, S_Q^d)$ to be \emph{noble} if $L(S_P^d, S_Q^d)$ is an $a,b$-Dyck path.

\begin{lemma}\label{goodRotNoble}
Every good sequence pair is $\rot$-conjugate to at least one noble sequence. 
\end{lemma}
\begin{proof}
Let $(S_P^d, S_Q^d)$ be a good sequence pair and $S^d_{P,Q} = (s_1, \ldots, s_d)$ be such that $s_i = \max(p_i, q_{i-1})$ where we interpret $q_0$ as $q_d$.  It will be convenient to also have a map $\gamma$ which reverses this as follows:
\[ \gamma(s_1, \ldots, s_d) = (S_P^d, S_Q^d)\]
where $p_i = s_i$ if $s_i > a/b$ and 0 otherwise, and $q_i = s_{i+1}$ if $s_{i+1} < a/b$ and 0 otherwise, interpreting $s_{d+1}$ as $s_1$.  

\textbf{Case 1:} $c=a$.  

In this case $s_1, \ldots, s_d$ is the zero sequence $(0,0, \ldots, 0)$ and $L(S_P^d, S_Q^d)$ is the valid Dyck path $N^{a}E^b$. 

\textbf{Case 2:} $a/b < c < a$.  

Let $L$ be the lattice path which starts at the origin and ends at $(2d, 2(s_1 + \cdots + s_d))$ given by
\[ L = N^{s_1}EN^{s_2}E \cdots N^{s_d}EN^{s_1}EN^{s_2}E \cdots N^{s_d}E.\]
Label the lattice points $P$ on $L$ with integers $w(P)$ as follows:  Label the origin 0.  If $P$ and $P'$ are consecutive lattice points, set $w(P') = w(P) - a$ if $P'$ is connected to $P$ by an $E$-step, and $w(P') = w(P) + b$ if $P'$ is connected to $P$ with an $N$-step.

By coprimality, there exists a unique lattice point on $L$ of minimal weight, $P_0$. Observe that by minimality, and the fact that $(s_1, \ldots, s_d)$ is not the zero sequence, $P_0$ must be immediately followed by a vertical run $N^{s_i}$ for some $1 \leq i \leq d$. Note: If $P_0$ is the terminal point of $L$ then we interpert the vertical run to be $N^{s_1}$.  

 If $i = 1$, then the entire path stays above the line $y = \frac{a}{b}x$ and it is clear that $L(S_P^d, S_Q^d)$ is a valid Dyck path.  Now suppose $i > 1$ and let $S = (s_{i-1}, s_i, \ldots, s_d, s_1, \ldots, s_{i-2})$.  The vertical run $N^{s_{i-1}}$ over the point $A_0$, immediately preceding $P_0$, must have height at most $a/b$.  Otherwise $A_0$ would have smaller weight, contradicting minimality. Thus $\gamma(S)$ is a very good sequence pair, so that $L(\gamma(S))$ makes sense.  We claim that $L(\gamma(S))$ is in fact a valid Dyck path so that $\gamma(S)$ is a noble sequence pair.  Consider the segmentation $L(\gamma(S)) = L_1\cdots L_{(b-1)/d}E$ where $L_i$ contains $d$ $E$ steps.  Since each segment is progressively further east, it will suffice to show that the final segment stays west of the line $y = \frac{a}{b}x$.  Since $(S_P^d, S_Q^d)$ is a good sequence pair, the copy of $P_0$ in $L_q$ stays west of the line $y=\frac{a}{b}x$. Since $P_0$ is minimal, no other point to the east of $P_0$ can cross the line $y=\frac{a}{b}x$.  Finally, since the vertical run immediately preceding $P_0$ has height at most $a/b$, we conclude that all of $L_q$ stays west of the line $\frac{a}{b}x$. 

\textbf{Case 3:} $0 \leq c < a/b$. 

Define $L$ as in Case 2, letting $P_0$ denote the lattice point of minimal weight.   $P_0$ is beneath a vertical run $N^{s_i}$ where $s_i > a/b$ since otherwise the point immediately following $P_0$ would be of smaller weight.  Let $S = (s_i, s_{i+1}, \ldots, s_d, s_1, \ldots, s_{i-1})$.  Since $s_i > a/b$, $\gamma(S)$ is a very good sequence pair so $L(\gamma(S))$ makes sense.  We claim that $L(\gamma(S))$ is a a valid Dyck path. To see this, consider the segmentation $L(\gamma(S)) = L_1 \cdots L_{(b-1)/d} N^c E$.  Since $c < a/b$, the point labeled $b-1$ stays west of the line $y=\frac{a}{b}x$.  Each segment $L_i$ is progressively further east, so it will again suffice to show that $L_q$ remains west of the line $y=\frac{a}{b}x$.  Since $(S_P^d, S_Q^d)$ is a good sequence pair, the copy of $P_0$ in $L_q$ stays west of the line $y=\frac{a}{b}x$, and since $P_0$ is minimal, no other point east of $P_0$ can cross the line $y = \frac{a}{b}x$. 

\begin{figure}[h]
\includegraphics[scale=.5]{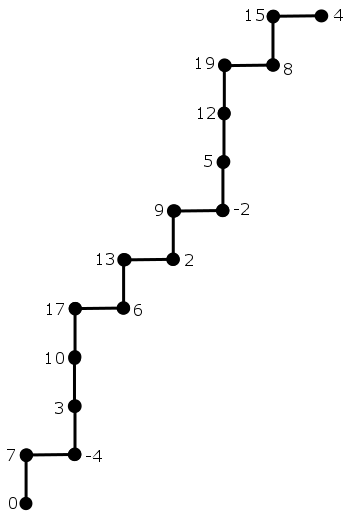}
\includegraphics[scale=.5]{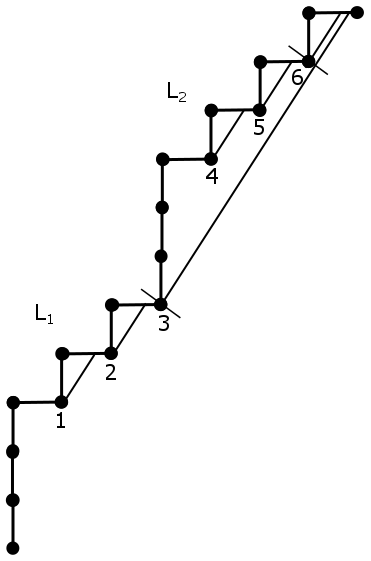}
\caption{}
\label{minP}
\end{figure}

For example, consider the good sequence pair $S_P^3 = (0,3,0)$ and $S_Q^3 = (0,1,1)$. Then we have $S_{P,Q}^3 = (1,3,1)$. The path on the left in Figure \ref{minP} shows the corresponding 11,7 lattice path $L$ with weight labels.  The point of minimal weight is labeled -4.  This appears at the bottom of a vertical run of length $3 > 11/7$.  Thus, $S = (3,1,1)$.  The path $L(\gamma(S))$ is shown on the right in Figure \ref{minP}.  The slashes indicate the segmentation into $L_1$ and $L_2$.  The final vertical run is $N^c = N^1$, and the result is a valid Dyck path. 
\end{proof}

\begin{lemma}\label{4.7}
Suppose that $(S_P^d, S_Q^d)$ is a noble sequence pair.  Then $\pi := \pi(L(S_P^d, S_Q^d)) \in NC_d(a,b)$ is noble and the $d$-modified $P$ and $Q$ rank sequences of $\pi$ are $S_P^d$ and $S_Q^d$. 
\end{lemma}
\begin{proof}
Define $S^d_{P,Q} = (s_1, \ldots, s_d)$ as in the proof of Lemma \ref{goodRotNoble} and let $c = a - q(s_1 + \cdots + s_d)$. The argument splits into cases: 

\textbf{Case 1:} $c > a/b$.

Let $L:= L(S_P^d, S_Q^d) = L_1L_2 \cdots L_qE$ where $L_1 = N^cEN^{s_2}E \cdots N^{s_d}E$ and $L_i = N^{s_1}E\cdots N^{s_d}E$ for $2 \leq i \leq q$. Since $(S_P^d, S_Q^d)$ is very good and $c > a/b$ we must have that the first entry of $S_P^d$ is 0.  Fix any index $1 \leq i \leq d$ such that $s_i > 0$ and any other index $1 \leq j \leq q-1$.  Both segments $L_j$ and $L_{j+1}$ of $L(s)$ contain a copy of the nonempty vertical run $N^{s_i}$.  First suppose $s_i > a/b$ and let $P_0$ and $P_1$ denote the points at the bottom of these vertical runs.  We have that $\ell(P_0)$ and $\ell(P_1)$ are rigid translations of one another, so the $P$ block visible from the copy of $N^{s_i}$ in $L_{j+1}$ is the image of the block visible from the copy of $N^{s_i}$ in $L_j$ under the operator $\rot^d$.  Now suppose $s_i < a/b$ and let $E_0$ and $E_1$ denote the east steps immediately following the vertical runs in $L_j$ and $L_{j+1}$.  The collection of lasers which hit $E_1$ are a rigid translation of the lasers which hit $E_0$, which implies that the $Q$ block determined by the lasers hitting $E_1$ is the image of the $Q$ block which results from lasers hitting $E_0$ in $L_j$ under the operator $\rot^d$.

Since the first entry of $S_P^d$ is 0 none of these blocks contain 1,  so the set of blocks not containing 1 is stable under $\rot^d$, which implies that the block containing 1 must be central, $\pi$ is invariant under $\rot^d$, and $\pi$ has no wrapping blocks.  Thus, $\pi$ is noble and the $d$-modified $P$ and $Q$ rank sequences of  $\pi$ are $S_P^d$ and $S_Q^d$. 

\textbf{Case 2:} $0 \leq c < a/b$.

As before, consider the segmentation $L(S_P^d, S_Q^d) = L_1 \cdots L_q N^c  E$ where $L_i = N^{s_1}E \cdots N^{s_d}E$.  Since $(S_P^d, S_Q^d)$ is very good and $0 < c < a/b$ we must have that the last entry of $S_Q^d$ is 0.  In this case, with the extra east step at the end of the path, lasers fired from the points at the bottom of consecutive copies of the vertical run $N^{s_i}$ in $L_j$ and $L_{j+1}$ are either 
\begin{enumerate}
\item translates of each other or 
\item they both hit $L$ on its terminal east step.  
\end{enumerate}
As described in case 1, $\rot^d$ invariance is guaranteed for all $P$ and $Q$ blocks determined by (1) and the fact that every such laser pair satisfies (1) or (2) implies there can be no wrapping blocks.  If $c = 0$ then there is no central block so we conclude $\pi$ is noble.  If $0 < c < a/b$ then all points $P$ such that $\ell(P)$ hits $L$ on its terminal east step are in a central $Q$ block containing $b-1$ so $\pi$ is noble in this case as well, and the $d$-modified $P$ and $Q$ rank sequences of  $\pi$ are $S_P^d$ and $S_Q^d$. 
\end{proof}

\begin{lemma}\label{4.8}
Suppose that $(P,Q) \in NC_d(a,b)$.  Then $(P,Q)$ is noble if and only if $(S_P^d, S_Q^d)$ is noble.
\end{lemma}
\begin{proof}
First suppose $(P,Q)$ is noble.  Let $S^d_{P,Q} = (s_1, \ldots, s_d)$, $c = a - \frac{b-1}{d}(s_1 + \cdots + s_d)$, $S_P = (p_1, \ldots, p_{b-1})$, and $S_Q = (q_1, \ldots, q_{b-1})$.  Since $(P,Q)$ contains no wrapping blocks,

\[ R(P,Q) = \begin{cases} (s_1, s_2, \ldots, s_d, s_1, \ldots, s_d, \ldots, s_1, \ldots, s_d, c) &\mbox{ if }c < a/b \\ (c, s_2, \ldots, s_d, s_1, \ldots, s_d, \ldots, s_1, \ldots, s_d, s_1) &\mbox{ if } c > a/b.\end{cases}\]
If $c = 0$ then $(S^d_P, S^d_Q)$ is automatically very good.  If $0 < c < a/b$ then the nobility of $(P,Q)$ implies $b-1$ is contained in the central block of $(P,Q)$.  Thus $q_d = 0$ so $(S^d_P, S^d_Q)$ is very good.  On the other hand, if $c > a/b$ then the nobility of $(P,Q)$ implies that 1 is in the central block so that $p_1 = 0$ and $(S^d_P, S^d_Q)$ is very good.  In both cases the vertical runs of $L(S^d_P,S^d_Q)$ agree with the rank sequence $R(P,Q)$, so $(S^d_P, S^d_Q)$ is noble.

Now suppose that $(P,Q)$ is not noble. The argument in the proof of Lemma 4.8 in \cite{RatCat} tells us that $P$ contains no wrapping blocks, and if it has a central block then it must contain 1.  If $P$ contains a central block with 1 then by Lemma \ref{p1} $Q$ cannot contain a wrapping block.  Thus, we may assume $P$ contains no central or wrapping blocks, which means $p_1 \neq 0$.  Since $(S_P, S_Q)$ is noble this means that $L:= L(S_{P,Q}^d) $ takes the form $ L_1L_2\cdots L_qN^cE$ where $L_i = (N^{p_1}EN^{\max(p_2, q_1)}E \cdots N^{\max(p_d, q_{d-1})}E$.  Suppose $Q$ contains a wrapping block $B$.  Let $f = \min(B)$ and $g = \max(B)$.  Since $B$ is wrapping, we have $1 \leq f \leq d$ and $b-d \leq g \leq b-1$.  Let $B'$ denote the inverse $d$-fold rotation of $B$. Then $\max(B') = b-d+f$ so that the $f^{th}$ entry of $S_Q^d$ is $\rank(B)$. Let $g'$ be the copy of $g$ contained in $L_1$ and $f'$ denote the copy of $f$ contained in $L_2$.  Then $\ell(g')$ hits the east step immediately following the vertical run above $f'$.  Let $P_0$ be the first point of $L_2$.  Since $g' \leq d \leq P_0$ and $\ell(g')$ hits an east step in $L_2$, this implies that $\ell(P_0)$ hits an east step in $L_2$ as well.  However, $L_1$ is a copy of $L_2$, so if we fire a laser from the initial point of $L_1$, the origin, then it must hit an east step of $L_1$, contradicting the fact that $L$ stays above the line $y=\frac{a}{b}x$. 

Finally, suppose $Q$ contains a central block which does not contain $b-1$.  By Lemma \ref{centralQb} this implies that $P$ contains a wrapping block so that $(S_P^d, S_Q^d)$ is not noble, a contradiction.
\end{proof}

\begin{theorem} The map $S^d : NC_d(a,b) \to \{ \mbox{good sequence pairs } (S^d_P, S^d_Q)\}$ is a bijection which commutes with the action of rotation.
\end{theorem}
\begin{proof}
By Lemma \ref{dRot}, $S^d$ commutes with rotation.  Now suppose $(S_P^d, S_Q^d)$ is a good sequence pair.  By Lemma \ref{goodRotNoble} it is conjugate to a noble sequence pair $(S_{P'}^d, S_{Q'}^d)$. By Lemma \ref{4.7} there exists $(P',Q') \in NC_d(a,b)$ such that $(P',Q')$ is noble and $S^d(P',Q') =  (S_{P'}^d, S_{Q'}^d)$.  As described in the proof of Lemma \ref{4.8}, this completely determines the rank sequence, and hence the vertical run sequence, of $(P',Q')$, which uniquely determines the partition. Therefore $(P',Q')$ is unique.  Since $S^d$ commutes with rotation, there must be a unique rotated partition pair which is the inverse image of $(S_P^d, S_Q^d)$, proving that $S^d$ is a bijection.
\end{proof}

Now we can enumerate the good sequence pairs.  To do this, begin by combining the pairs of sequences into a single sequence $(s_0, s_1, \ldots, s_d)$ of length $d+1$ as follows:

\[ s_i = \begin{cases}  \max(p_{i+1}, q_i) \mbox{ if } 1 \leq i \leq d-1 \\ \max(p_1, q_d) \mbox{ if } i=d \end{cases}\]

This transformation is a bijection onto the set of nonnegative integer sequences of length $d$ whose entries sum to at most $ad/(b-1)$, which are counted by 
\[ { \lfloor ad/(b-1)\rfloor+ d \choose d}.\]

\begin{corollary}\label{numDInv}
Let $a$ and $b$ be coprime positive integers and $d | (b-1)$.  The number of $(P,Q) \in NC(a,b)$ which are invariant under $\rot^d$ is given by
\[ { \lfloor ad/(b-1)\rfloor+ d \choose d}.\]
\end{corollary}

\begin{corollary}\label{numKDInv}
Let $a$ and $b$ be coprime positive integers and $d | (b-1)$. Let $p$ be a nonnegative integer such that $\frac{b-1}{d}p \leq a$. The number of $(P,Q) \in NC_d(a,b)$ with a central block in either $P$ or $Q$ and $p$ orbits of non-central blocks under the action of $\rot^d$ is 
\[ {d \choose p} {\lfloor \frac{ad}{b-1}\rfloor - 1 \choose p}.\]
The number of $(P,Q) \in NC_d(a,b)$ with no central block and $p$ orbits of noncentral blocks under the action of $\rot^d$ is
\[ \begin{cases}  {d \choose p} {\lfloor \frac{ad}{b-1}\rfloor - 1 \choose p - 1} & \mbox{ if } \frac{b-1}{d} | a \\
0 & \mbox{ if } \frac{b-1}{d} \not| a.\end{cases}\]
\end{corollary}

\begin{corollary}\label{multiKrew}
Let $a$ and $b$ be coprime positive integers and $d | (b-1)$.  Let $m_1, \ldots, m_a$ be nonnegative integers which satisfy $\frac{b-1}{d}(m_1 + 2m_2 + \cdots + am_a) \leq a$.  The number of $(P,Q) \in NC(a,b)$ which are invariant under $\rot^d$ and have $m_i$ orbits of noncentral blocks of rank $i$ under the action of $\rot^d$ is 

\[ {d \choose m_1, m_2, \ldots, m_a, d-m}\]
where $m = m_1 + m_2 + \cdots + m_a$.
\end{corollary}

\section{Cyclic Sieving}

Let $X$ be a finite set, $C = \langle c \rangle$ be a finite cyclic group acting on $X$, $X(q) \in \N[q]$ be a polynomial with nonnegative integer coefficents, and $\zeta \in \C$ be a root of unity with multiplicative order $|C|$.  The triple $(X, C, X(q))$ exhibits the \emph{cyclic sieving phenomenon} if for all $d \geq 0$ we have $X(\zeta^d) = |X^{c^d}| = |\{x \in X | c^d.x = x\}|$. For additional background and examples of the cyclic sieving phenomenon in other contexts, see \cite{CSP}.  

\begin{theorem}
Let $a$ and $b$ be coprime and $\mathbf{r} = (r_1, r_2, \ldots, r_a)$ be  sequence of nonnegative integers satisfying $r_1 + 2r_2 + \cdots + ar_a = a$.  Set $k = \sum_{i=1}^a r_i$.  Let $X$ be the set of $(P,Q) \in NC(a,b)$ with $r_i$ blocks of rank $i$, where a block may come from either $P$ or $Q$. Then the triple $(X, C, X(q))$ exhibits the cyclic sieving phenomenon, where $C = \Z_{b-1}$ acts on $X$ by rotation and 
\[ X(q) = \Krew_q(a,b,\mathbf{r}) = \frac{ [b-1]!_q}{[r_1]!_q \cdots [r_a]!_q[b-k]!_q}\]
is the $q$-rational Kreweras number.
\end{theorem}
\begin{proof}
The proof of Theorem 5.1 in \cite{RatCat} shows that $\Krew_q(a,b,\mathbf{r})$ is a polynomial in $q$ with nonnegative integer coefficients and evaluates to the multinomial coefficient given in Corollary \ref{multiKrew}.
\end{proof}

\begin{theorem}
Let $a$ and $b$ be coprime, $1 \leq k \leq a$, and $X$ be the set of $(P,Q) \in NC(a,b)$ with $k$ blocks in total. The triple $(X, C, X(q))$ exhibits the cyclic sieving phenomenon where $C = \Z_{b-1}$ acts on $X$ by rotation and 
\[ X(q) = \Nar_q(a,b,k) = \frac{1}{[a]_q} {a \brack k}_q {b-1 \brack k-1}_q\]
is the $q$-rational Narayana number.
\end{theorem}
\begin{proof}
As can be read off from Theorem 5.2 in \cite{RatCat}, the root of untiy evaluation agrees with the formula given in Corollary \ref{numKDInv}.
\end{proof}

\begin{theorem}\label{basicCSP}
Let $a$ and $b$ be coprime, $X$ be the set of $(P,Q) \in NC(a,b)$ and
\[ X(q) = \Cat_q(a,b) = \frac{1}{[a+b]_q} {a+b \brack a,b}_q\]
be the $q$-rational Catalan number.  Then the triple $(X,C,X(q))$ exhibits the cyclic sieving phenomenon, where $C=\Z_{b-1}$ acts by rotation.
\end{theorem}
\begin{proof}
This follows from the fact that $\Cat_q(a,b)$ evaluates to the expression given in Corollary \ref{numDInv} when we let $q \to e^{2\pi i d/(b-1)}$.
\end{proof}

The special case of $(a,b) = (n+1, n)$ was considered by Thiel in \cite{Marko}, and this instance of the cyclic sieving phenomenon was proven for this case.  A simple bijection relates $n+1, n$-Dyck paths (and therefore elements of $NC(n+1, n)$) to the noncrossing $(1,2)$-configurations, a variant of one of the hundreds of Catalan objects listed in Stanley's Catalan addendum. \cite{Stanley}.  For convenience, we reprint the relevant definitions here:
\begin{defn}(Thiel)
Call a subset of $[m]$ a \emph{ball} if it has cardinality 1 and an \emph{arc} if it has cardinality 2.  Define a (1,2)-configuration on $[m]$ as a set of pairwise disjoint balls and arcs. Say that a (1,2)-configuration $F$ has a crossing if it contains arcs $\{i_1, i_2\}$ and $\{j_1, j_2\}$ with $i_1 < j_1 < i_2 < j_2$.  If $F$ has no crossing, it is called noncrossing. Define $X_n$ to be the set of noncrossing (1,2)-configurations on $[n-1]$.
\end{defn}

\begin{proposition}
There is a bijection $\tau$ between $n+1, n$-Dyck paths and $X_n$ that commutes with the action of rotation. 
\end{proposition}
\begin{proof}
Given an $n+1, n$-Dyck path $D$, define $\tau(D)$ as follows:  Read the labels from 1 to $n-1$. If the point labeled $i$ does not fire a laser, leave it unmarked in the $(1,2)$ configuration.  Otherwise, it fires a laser which hits an east step with left endpoint whose $x$-coordinate is $j$.  If $i = j$, decorate $i$ with a dot.  Otherwise, draw an arc from $i$ to $j$.  For the reverse map, note that there is a unique way to fire a laser from $i$ in such a way that it hits an east step with left endpoint having $x$-coordinate $j$.  To see that this commutes with rotation, it will be easiest to think in terms of noncrossing partitions.  In particular, given $(P,Q) \in NC(n+1, n)$ we obtain its corresponding $(1,2)$ configuration as follows:  For each block $B$ of $P$, draw an arc from $\min(B) - 1$ to $\max(B)$. If $\min(B) = 1$, draw an arc from $n-1$ to $\max(B)$.  Each nontrivial block of $Q$ will be a singleton $\{i\}$ of rank 1.   Draw a ball at $i$. For an $n+1, n$-Dyck path $D$, this construction bijects $\pi(D)$ to $\tau(D)$.  It follows that rotation of $(P,Q)$ simply rotates its associated $(1,2)$ configuration. 
\end{proof}

Hence, Theorem \ref{basicCSP} specializes to Thiel's result when $(a,b) = (n+1, n)$.  Figure \ref{thiel} shows an example of a 7,6-Dyck path and its corresponding noncrossing (1,2) configuration.  
\begin{figure}[h]
\includegraphics[scale=.5]{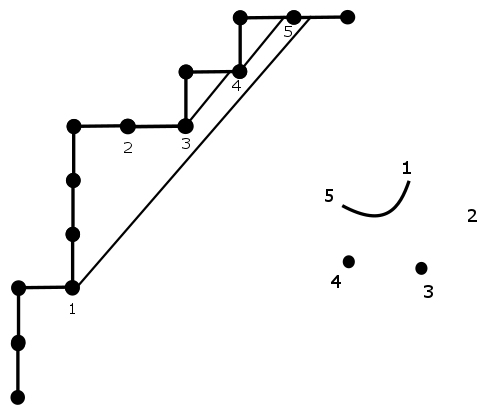}
\caption{A 7,6-Dyck path and its corresponding noncrossing (1,2) configuration in $X_6$}
\label{thiel}
\end{figure}

\section{Parking Functions}

Let $W$ be a finite Coxeter group with root lattice $Q$ and Coxeter number $h$.  In \cite{ParkingSpaces}, Armstrong, Reiner, and Rhoades defined \emph{$W$-parking functions}, a generalization of the classical type $A$ parking functions, as elements of $Q / (h+1)Q$ which carry an action of $W$ called the standard parking space.  They went on to show the $W$-parking functions in fact carry an action of $W \times C$ where $C$ is the cyclic subgroup of $W$ generated by a Coxeter element.  In \cite{ParkingStructures}, Rhoades defines $k-W$-parking spaces, a Fuss analog of their work.  In \cite{RatCat}, the author and Rhoades provide a rational extension $\Park^{NC}(a,b)$ of \cite{ParkingSpaces} and \cite{ParkingStructures} when $W$ is the symmetric group $\mathfrak{S}_a$.  Though rational parking functions have been studied elsewhere in the literature \cite{compShuffle}, the action of $\mathfrak{S}_a \times \Z_{b-1}$ on parking functions had previously only been known in the case $a < b$.  Here, we generalize to all coprime $a$ and $b$.  We begin with basic definitions.

For all coprime $a$ and $b$, we define an \emph{$a,b$-noncrossing parking function} as a pair $((P,Q), f)$ where $(P,Q) \in NC(a,b)$ and $f : \{B | B \in P \mbox{ or } B \in Q\} \to 2^{[a]}$ is a labeling of blocks of $P$ and $Q$ such that the following holds:
\begin{itemize}
\item $[a] = \bigsqcup_{B \in P \mbox{ or } B \in Q} f(B)$
\item for all blocks $B$ we have 
\[ |f(B)| = \begin{cases} \rank_P(B) & \mbox{ if $B \in P$}\\ \rank_Q(B) & \mbox{ if $B \in Q$.} \end{cases}\]
\end{itemize}
Alternatively, we can view this as a labeling of the $N$ steps of an $a,b$ Dyck path by the numbers 1 through $a$, where the labels increase as one moves up a vertical run.  We will refer to the set of all $a,b$-noncrossing parking functions as $\Park^{NC}(a,b)$.
\begin{proposition}\label{action}
$\Park^{NC}(a,b)$ carries an action of $\mathfrak{S}_a \times \Z_{b-1}$ where $\mathfrak{S}_a$ permutes block labels and $\Z_{b-1}$ rotates blocks.
\end{proposition}
\begin{proof}
Rotation preserves vertical lengths, and thus ranks, by the definition of $\rot$ and Proposition \ref{rotPreserve}. 
\end{proof}

We would like to state a character formula for the action described in Proposition \ref{action}.  Let $V = \C^a/\langle (1, \ldots, 1)\rangle$ be the reflection representation of $\symm_a$ and $\zeta = e^{\frac{2 \pi i}{b-1}}$.  Recall from \cite{Okada} that  the map $\phi: \mathfrak{S}_a \to \C$ by $\phi(w) = b^{\dim V^w}$ is a character of $\symm_a$ if and only if $\gcd(b, a) = 1$.  In particular, this suggests that we should have a single formula which doesn't depend on whether or not $a < b$.  Given $w \in \symm_a$ and $d \geq 0$, let $\mult_w(\zeta^d)$ be the multiplicity of $\zeta^d$ as an eignevalue in the action of $w$ on $V$.  With this notation, we have the following for the character $\chi$:

\begin{theorem}\label{park}
Let $w \in \symm_a$ and $g$ be a generator of $\Z_{b-1}$.  Then we have
\[ \chi(w, g^d) = b^{\mult_w(\zeta^d)} \tag{1}\label{eq1}\]
for all $w \in \symm_a$ and $d \geq 0$. 
\end{theorem}
\begin{proof}
If $d | (b-1)$ we have

\[ \mult_w(\zeta^d) = \begin{cases} \#(\mbox{cycles of } w) - 1 & \mbox{ if } q=1 \\ \#(\mbox{cycles of $w$ of length divisible by $q$}) & \mbox{ otherwise}\end{cases} \tag{2}\label{eq2}\]
where $q = \frac{b-1}{d}$. To see this, first suppose $q = 1$.  Then $d = b-1$ so $\zeta^d = 1$. The vectors which are fixed by the permutation matrix of $w$ are precisely those which are constant on cycles of $w$. Deleting the all 1's vector leaves us with $\#(\mbox{cycles of } w) - 1$ such linearly independent vectors.  Now suppose $q > 1$.  Then the vectors which increase by a factor of 0 or $\zeta^d$ along cycles of length divisible by $q$, and which are 0 along cycles of length not divisible by $q$, are the eigenvectors with eigenvalue of $\zeta^d$.  Each cycle of length divisible by $q$ contributes one such eigenvector. 

We are now ready to count the number of $a,b$-noncrossing parking functions which are fixed under the action of $(w, g^d)$.  We will handle the cases $q=1$ and $q > 1$ separately.

\textbf{Case 1:} $q=1$.  In this case, $g^d = g^{b-1} = 1$ so we can ignore the action of $\Z_{b-1}$ and just consider elements of $\Park^{NC}(a,b)$ which are fixed by $w \in \symm_a$. To do this, we will construct an equivariant bijection $f: \Park^{NC}(a,b) \to S$ and show that $S$ has the desired character.  

 Let $\Park_{a,b}$ be the set of sequences $(p_1, \ldots, p_a)$ of positive integers whose nondecreasing rearrangement $(p_1' \leq p_2' \leq \cdots \leq p_a')$ satisfies $p_i' \leq \frac{b}{a}(i-1)+1$.  These are called \emph{rational slope parking functions}.  Our choice of $S$ will be $\Park_{a,b}$.  Then $\symm_a$ acts on $\Park_{a,b}$ by 
\[ w.(p_1, \ldots, p_a) = (p_{w(1)}, \ldots, p_{w(a)}).\]

In particular, $w$ fixes precisely those parking functions which are constant on cycles of $w$.  Let $c_w$ denote the number of cycles of $w$ and $\chi(w)$ denote the character of the action.  There are $b^{c_w}$ sequences of length $a$ which are constant on cycles of $w$.  By the cycle lemma, exactly one cyclic rotation of each of these will be a valid rational slope parking function, so we have

\[ \chi(w) = b^{c_w - 1} = b^{\mult_w(1)}. \tag{3}\label{eq3}\]
We are now left to build our equivariant bijection $\phi: \Park^{NC}(a,b) \to \Park_{a,b}$.  Let $((P,Q), f)$ be an $a,b$-noncrossing parking function.  Define $\phi((P,Q), f) = (p_1, \ldots, p_a)$ by 
\[ p_i = \begin{cases} \min(B) & \mbox{ if } B \in P \mbox{ and } i \in f(B) \\ \max(B) + 1 & \mbox{ if } B \in Q \mbox{ and } i \in f(B).\end{cases} \tag{4}\label{eq4}\]
Equivalently, one can think of the pair $((P,Q), f)$ as an $a,b$-Dyck path where the north steps are labeled by the numbers 1 through $a$ and each vertical run has increasing labels.  The underlying dyck path $D$ is such that $\pi(D) = (P,Q)$, and the labels on a particular vertical run that determine the rank of a block $B$ are given by $f(B)$.  

\begin{example}\label{parkEx1}
Consider the labeled 9,4-Dyck path shown in Figure \ref{park94}. This corresponds to the partitions $P = \{\{1,2,3\}\}$ with rank 3 and $Q = \{\{1\}, \{2\}, \{3\}\}$ each with rank 2, and the function $f$ defined by $f(\{1,2,3\}) = \{3, 5, 6\}$, $f(\{1\}) = \{1,8\}$, $f(\{2\}) = \{4,9\}$, and $f(\{3\}) = \{2,7\}$.  The associated rational slope parking function $(p_1, \ldots, p_a)$ is $(2, 4, 1, 3, 1, 1, 4, 2, 3).$  This can be read off from $((P,Q),f)$ via equation \ref{eq4} or by setting $p_i$ equal to 1 greater than the $x$ coordinate of the north step labeled by $i$.  From this point of view, and the fact that $D$ must stay above the line $y = \frac{a}{b}x$, we see that for all $i$, we must have $\frac{i-1}{p_i' - 1} \geq \frac{a}{b}$ which is equivalent to the condition that $p_i' \leq \frac{b}{a}(i-1)+1$. In other words, $(p_1, \ldots, p_a)$ is indeed a sequence in $\Park_{a,b}$.
\end{example}

\begin{figure}[h]
\includegraphics[scale=.5]{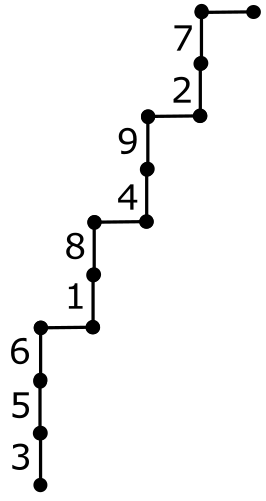}
\caption{}
\label{park94}
\end{figure}

Conversely, suppose we are given $(p_1, \ldots, p_a)$.  Let $n_i$ denote the number of entries of $(p_1, \ldots, p_a)$ which are equal to $i$.  Then we can recover the labeled Dyck path $D$ by setting 
\[ D = N^{n_1}EN^{n_2}E \cdots N^{n_a}E\]
and labeling the vertical run with $x$-coordinate $i - 1$ by the numbers in $\{i \,|\, p_i = 1\}$ in increasing order.  Since permuting labels $i$ and $j$ on the Dyck path corresponds to swapping $p_i$ and $p_j$, we conclude that $\phi$ is in fact an equivariant bijection.  Since $\phi$ preserves character formulas, equation \ref{eq3} implies that $\chi(w, g^d) = b^{\mult_w(\zeta^d)}$ when $q=1$. 

\textbf{Case 2:} $q > 1$. Let $r_q(w)$ denote the number of cycles of $w$ having length divisible by $q$. We will show that 
\[ | \Park^{NC}(a,b)^{(w,g^d)}| = | \{p \in \Park^{NC}(a,b) | (w,g^d).p = p\}| = b^{r_q(w)}.\tag{5}\label{eq5}\]

To do this, we will show that both sides count a certain set of functions.  First, define an action of $g$ on $[b-1] \cup \{0\}$ by the permutation $(1, 2, \ldots, b-1)(0)$.  We say a function $e: [a] \to [b-1] \cup \{0\}$ is $(w, g^d)$-equivariant if 
\[ e(w(j)) = g^d e(j) \tag{6}\label{eq6}\]
 for all $1 \leq j \leq a$. To count such functions, we first consider what happens on cycles of $w$.  By equation \ref{eq6}, if $e(k) \neq 0$ then we have $e(w(k)) = e(k) + d$, where addition is performed modulo $b-1$. Thus, the values $e$ takes on a cycle are completely determined by the value taken on one element of that cycle.  Further, if a cycle has length not divisible by $q$ then equation \ref{eq6} forces $e(k) = 0$ for any $k$ in that cycle. Thus, the number of $(w, g^d)$-equivariant functions is $b^{r_q(w)}$. 

For example, let $(a,b) = (14, 13)$, $q = 3$, and consider 
\[w = (5, 1, 8)(2, 3, 6, 7, 9, 10)(4, 11)(12, 13, 14),\]
 written in cycle notation.  Let $e$ be the function defined by 
\[(e(1), e(2), \ldots e(14)) = (9, 1, 5, 0, 5, 9, 1, 1, 5, 9, 0, 2, 6, 10).\]
Note, for instance, how $w(5) = 1$ and $9 = e(1) = 5 + 4 = e(5) + d$. In this example, $e$ is indeed a $(w, g^d)$-equivariant function.

Next, we count equivariant functions according to their fiber structure. We say a set partition $\sigma = \{B_1, B_2, \ldots \}$ of $[a]$ is \emph{$(w,q)$-admissible} if the following conditions hold:
\begin{enumerate}
\item $\sigma$ is $w$-stable. ie, $w(\sigma) = \{w(B_1), w(B_2), \ldots \} = \sigma$
\item There is at most one block $B_{i_0}$ such that $w(B_{i_0}) = B_{i_0}$
\item For any block $B_i$ which is not $w$-stable, the blocks 
\[B_i, w(B_i), \ldots, w^{q-1}(B_i)\]
 are pairwise distinct, and we have $w^q(B_i) = B_i$. 
\end{enumerate}
Given a $(w, g^d)$-equivariant function $e$, define a set partition by $\sigma$ by $i \sim j$ if and only if $e(i) = e(j)$. For instance, consider the example given above.  Then we have 
\[\sigma = \{\{4, 11\}, \{2, 7, 8\}, \{12\}, \{3, 5, 9\}, \{13\}, \{1, 6, 10\}, \{14\}\}.\]

 In general, $\sigma$ is $w$-stable because if $i \neq 0$ and $B = e^{-1}(i)$ then $w(B) = e^{-1}(i + d)$, and if $i = 0$ then $w(B) = B$. Furthermore, $e^{-1}(0)$ is the only block which is fixed, so $(2)$ is satisfied.  Lastly, since $d, 2d, \ldots, (q-1)d$ are distinct, this means $e^{-1}(i), w(e^{-1}(i)), \ldots, w^{q-1}(e^{-1}(i))$ are distinct, and $w^q(e^{-1}(i)) = e^{-1}(i) + qd = e^{-1}(i)$ since arithmetic is performed modulo $b-1$.  

Each $w$-stable orbit is of size $q$ or size 1, depending on whether its blocks come from the inverse image of nonzero numbers or not. Given a particular fiber structure and $w$, consider how many $(w,g^d)$-equivariant functions could give rise to such a structure. There are $b-1$ choices for how to map some element of the first orbit. In our example, given the orbit containing $\{2, 7, 8\}$, $\{3, 5, 9\}$, and $\{1, 6, 10\}$, we have $b-1$ ways to assign $e(2)$, which then forces $e(7) = e(2)$, $e(8) =e(2)$, $e(3) = e(2) + d$, and so on.  Once this choice is made, the value of $e$ is determined for all elements in the orbit. Since orbits are of size $q$, this eliminates $q$ possible assignments from the next orbit we consider.  In our example, this would give us $b-1 - q = 9$ choices for $e(12)$ in the orbit $\{12\}, \{13\}, \{14\}$. More generally, if we let $t_{\sigma}$ denote the number of $w$-orbits in $\sigma$ of size $q$ then we have
\[ (b-1)(b-1-q) \cdots (b-1- (t_\sigma - 1)q) \]
 $(w, g^d)$-equivariant functions corresponding to a $(w, q)$-admissible set partition $\sigma$.  Thus, there are
\[ \sum_{\substack{\sigma \mbox{ a }(w,q)-\mbox{admissible} \\ \mbox{partition}}} (b-1) (b-1-q) \cdots (b-1- (t_\sigma - 1)q) = b^{r_q(w)} \tag{7}\label{eq7}\]
$(w, g^d)$-equivariant functions. 

To relate this back to parking functions, fix $((P, Q) , f) \in \Park^{NC}(a,b)$. Let $\tau((P,Q), f)$ be the set partition of $[a]$ defined by $i \sim j$ if and only if $i, j \in f(B)$. In Example \ref{parkEx1} we recover the set partition $\tau((P,Q), f) = \{\{3,5,6\}, \{1,8\}, \{4,9\}, \{2,7\}\}$.  More generally, if $((P, Q), f)$ is an element of $\Park^{NC}(a,b)^{(w, g^d)}$ then $\tau(\pi, f)$ is a $(w,q)$-admissible set partition.  It is $w$-stable because if $i \sim j$, then $i \sim j$ after we apply $g^d$, so $w$ must also keep $i$ and $j$ in the same block. There is at most one central block in $(P,Q)$ so at most one $B$ such that $w(B) = B$. Finally, if $((P, Q), f) \in \Park^{NC}(a,b)^{(w, g^d)}$ then $w$ behaves like $\rot^{-d}$ on $(P,Q)$, which proves that $\tau((P,Q), f)$ is $(w,q)$-admissible. 

Given a $(w,q)$-admissible partition $\sigma$ of $[a]$, we will count how many $((P,Q), f) \in \Park^{NC}(a,b)^{(w, g^d)}$ are such that $\tau((P,Q),f) = \sigma$. We begin by constructing the underlying $a,b$-noncrossing partition pair $(P,Q)$.  If $\sigma$ has $m_i$ non-singleton $w$-orbits of blocks of size $i$, then $(P,Q)$ must have $m_i$ $\rot^d$-orbits of non-central blocks of rank $i$. By Corollary \ref{multiKrew}, there are 
\[ {d \choose m_1, \ldots, m_a, d- t_\sigma }\]
such $(P,Q) \in NC_d(a,b)$. It now only remains to define $f$. The $\rot^d$-orbits of noncentral blocks of $(P,Q)$ of rank $i$ must be paired with nonsingleton $w$-orbits of blocks of $\sigma$ of size $i$. For each $i$, there are $m_i!$ ways to perform this matching. Each orbit has size $q$, so there are $q$ ways to choose which block determines labeling of the first blocks in a noncentral $\rot^d$ orbit. Thus, the number of $((P,Q), f) \in  \Park^{NC}(a,b)^{(w, g^d)}$ such that $\tau((P,Q), f) = \sigma$ is given by 
\begin{align*}
q^{m_1}q^{m_2} \cdots q^{m_a} m_1! m_2! \cdots m_a! &{d \choose m_1, m_2, \ldots, d-t_\sigma} \\
&= q^{m_1}q^{m_2} \cdots q^{m_a} \frac{d!}{(d-t_\sigma)!} \\
&= q^{t_\sigma} d(d-1)(d-2) \cdots (d-(t_\sigma - 1)) \\
&= (b-1)(b-1-q) \cdots (b-1- (t_\sigma - 1)q).
\end{align*}
Summing over all $(w,q)$-admissible partitions gives equation \ref{eq7}, so we conclude that \ref{eq5} holds as desired.
\end{proof}

Theorem \ref{park} can be used to generalize Theorem 6.3 in \cite{RatCat} to all coprime $a$ and $b$. In particular, we obtain a rational analog of the Generic Strong Conjecture of \cite{Evidence} in type $A$ for any coprime pair $(a,b)$. Following the definitions and terminology given in \cite{Evidence}, the following holds:

\begin{theorem}
Let $\mathcal{R} \subset \mathsf{Hom}_{\C[\symm_a]}(V^*, \C[V]_b)$ denote the set of \newline $\Theta \in \mathsf{Hom}_{\C[\symm_a]}(V^*, \C[V]_b)$ such that the parking locus $V^\Theta(b) \subset V$ cut out by the ideal 
\[\langle \Theta(x_1)-x_1, \ldots, \Theta(x_{a-1})-x_{a-1}\rangle \subset \C[V]\]
 is reduced, where $x_1, \ldots, x_{a-1}$ is any basis of $V^*$.  For any $\Theta \in \mathcal{R}$, there exists an equivariant bijection of $\symm_a \times \Z_{b-1}$-sets 
\[ V^\Theta(b) \simeq_{\symm_a \times \Z_{b-1}} \Park^{NC}(a,b).\]
There also exists a nonempty Zariski open subset $\mathcal{U}\subseteq \mathsf{Hom}_{\C[\symm_a]}(V^*, \C[V]_b)$ such that $\mathcal{U} \subseteq \mathcal{R}$. 
\end{theorem}

The proof is again a recreation of sections \cite[Sections 4, 5]{Evidence}. The only difference now is that we replace the reference to the proof of \cite[Lemma 8.5]{ParkingStructures} in the proof of \cite[Lemma 4.6]{Evidence} with the corresponding argument in the proof of Theorem \ref{park}. 

\section{Future Work}

The next step in this research is to generalize the results given here to other reflection groups.  Given a reflection group $W$, Reiner \cite{Reiner} defined a $W$-noncrossing partition which reduces precisely to our notion of noncrossing partition when $W = \symm_a$ and $(a,b) = (a, a+1)$.   Explicitly, let $\abs(W)$ denote the poset of $W$ under the absolute order.  We define the poset of noncrossing partitions of $W$ by 
\[ NC(W,c) = [1,c]\]
where $c \in W$ is a Coxeter element. Since $[1,c] \cong [1,c']$ for any choice of Coxeter elements $c$ and $c'$, we may simply write $NC(W)$.  When $W = \symm_a$, we have that $NC(\symm_a)$ is just the usual poset of noncrossing partitions of $[a]$, ordered by refinement.  When $(a,b) = (n, kn+1)$ for some positive integer $k$, $NC(a,b)$ consists of noncrossing partitions with block sizes divisible by $k$.  Armstrong \cite{memoirs} studied a Fuss-Catalan version of these $k$-divisible noncrossing partitions which made sense for any reflection group $W$ by considering $k$-multichains in the lattice of noncrossing partitions $NC(W)$.  It is then natural to ask whether the results obtained here for rational noncrossing partitions may be generalized to other reflection groups.  In type $B$, where $W$ is the group of signed permutations, the combinatorial model \cite[Section 6]{ParkingSpaces}  for noncrossing partitions is the centrally symmetric partitions of $\pm [n]$, those for which at most one block is sent to itself by $n$-fold rotation.  Central symmetry is a concept that makes sense even for rational noncrossing partitions, so there is hope to extend these results to the rational case in type B.  However, it is less clear what to do for the other Weyl groups, and would be nice to have a uniform approach for defining and working with rational noncrossing partitions for any Weyl group.

\end{document}